\documentclass[
                final,             
               onefignum,
               onetabnum
                        ]{siamart250211}
\usepackage{amsmath}
\usepackage{upref} 
\usepackage{amssymb}
\usepackage{breakcites}
\usepackage{color}
\usepackage{array}
\usepackage{calc}

\usepackage{algorithm}
\usepackage{algpseudocode}
\usepackage{float} 

\usepackage{graphicx}
\usepackage{url}
\usepackage{placeins} 
\usepackage{fancyhdr,substr}  
\usepackage{gitinfo}

\usepackage[english]{babel}
\usepackage{cleveref}

\usepackage{tabularx}
\usepackage{booktabs}
\numberwithin{table}{section}
\usepackage{multirow} 

\usepackage{subcaption}
\numberwithin{figure}{section}

\ifpdf
  \DeclareGraphicsExtensions{.eps,.pdf,.png,.jpg}
\else
  \DeclareGraphicsExtensions{.eps}
\fi

\makeatletter
\def\@cite#1#2{[{#1\if@tempswa ,~#2\fi}]}
\makeatother
\headers{A scaling and recovering algorithm}{Awad H. Al-Mohy and Xiaobo Liu}
\title{A scaling and recovering algorithm for the matrix $\varphi$-functions%
       \thanks{Version of \today.
       \funding{This work was supported by the Deanship of Scientific Research at King Khalid University Research Groups Program (grant RGP. 1/318/45).
               }
              }
      }

\author{Awad H. Al-Mohy%
           \thanks{%
                   Department of Mathematics, King Khalid University, Abha,
                   Saudi Arabia
                   (\email{ahalmohy@kku.edu.sa}
                    ).
        }
        \and Xiaobo Liu%
           \thanks{%
            Max Planck Institute for Dynamics of Complex Technical Systems,
           	Magdeburg, 39106, Germany
           	(\email{xliu@mpi-magdeburg.mpg.de}).
           	}
}

\makeatletter
\def\argmin{\mathop{\operator@font argmin}}
\makeatother

\def\expm{\t{expm}}
\def\phifunm{\t{phi\_funm}}

\def\phipade{\t{phipade}}
\def\phipadeopt{\t{phipade\_opt}}
\def\phipadedft{\t{phipade\_dft}}
\def\expmblktri{\t{expm\_blktri}}

\def\Ht{\widetilde{h}}

\def\dt{\widetilde{d}}

\def\phat{\widehat{p}}

\def\C{\mathbb{C}}

\def\and{\mathop{\mathrm{and}}}

\def\Htmd{\Ht@@'_{2m+1}}

\def\Pant{Pad\'e approximant}

\mathcode`@="8000 
{\catcode`\@=\active\gdef@{\mkern1mu}}

\def\be{backward error}

\def\alg{algorithm}
\def\Alg{Algorithm}
\def\e{eigenvalue}
\def\t#1{\texttt{\upshape #1}}

\DeclareMathOperator{\spn}{span}
\def\iu{\ensuremath{\mathrm{i}}} 

\newcounter{mylineno}
\makeatletter
\let\oldtabcr\@tabcr

\def\mynewline{\refstepcounter{mylineno}%
                \llap{\footnotesize\arabic{mylineno}\hspace{5pt}}%
               }

\gdef\@tabcr{\@stopline \@ifstar{\penalty%
            \@M \@xtabcr}\@xtabcr\mynewline}

\makeatother
\def\a{\alpha}

\def\D{\Delta}

\def\th{\theta}

\def\resp{respectively}

\def\phat{\widehat{p}}

\def\dt{\mathrm{d}}
\def\1i{\mathrm{i}}
\def\e{\mathrm{e}}

\def\norm#1{\|#1\|}
\def\normt#1{\|#1\|_2}

\def\normi#1{\|#1\|_1}
\def\normo#1{\|#1\|_{\infty}}

\def\nbyn{n \times n}

\def\C{\mathbb{C}}
\def\R{\mathbb{R}}

\def\N{\mathbb{N}}

\def\cL{\mathcal{D}}
\def\cR{\mathcal{R}}

\let\oldref\ref
\def\ref#1{{\normalfont\oldref{#1}}}
\def\eqref#1{{\normalfont(\oldref{#1})}}

\newcounter{exprmt}
\def\Experiment{Experiment}

               {\begin{list}{\indent\emph{\Experiment} {\upshape\arabic{exprmt}.}}%
                {\usecounter{exprmt}
                \setlength{\leftmargin}{\rightmargin}
                \setlength{\labelwidth}{\leftmargin}
                \addtolength{\labelwidth}{-\labelsep}
                \setlength{\topsep}{0in}
                \setlength{\itemsep}{0pt}
                \setlength{\listparindent}{\parindent}%
               }}%
               {\end{list}}
\mathchardef\Gamma="7100 \mathchardef\Delta="7101
\mathchardef\Theta="7102 \mathchardef\Lambda="7103
\mathchardef\Xi="7104 \mathchardef\Pi="7105 \mathchardef\Sigma="7106
\mathchardef\Upsilon="7107 \mathchardef\Phi="7108
\mathchardef\Psi="7109 \mathchardef\Omega="710A

\newsiamthm{myalg}{\textbf{Algorithm}}
\begin{document}
\maketitle

\begin{center}
	\textit{In memory of Nick Higham, an enduring inspiration to generations}
\end{center}
\smallskip

\begin{abstract}
A new scaling and recovering algorithm is proposed for simultaneously computing the matrix $\varphi$-functions that arise in exponential integrator methods for the numerical solution of certain first-order systems of ordinary differential equations. The algorithm initially scales the input matrix down by a nonnegative integer power of two, and then evaluates the $[m/m]$ diagonal Pad\'e approximant to $\varphi_p$, where $p$ is the largest index of interest. The remaining $[m+p{-}j/m]$ Pad\'e approximants to $\varphi_j$, $0 \le j < p$, are obtained implicitly via a recurrence relation. The effect of scaling is subsequently recovered using the double-argument formula. A rigorous backward error analysis, based on the $[m+p/m]$ Pad\'e approximant to the exponential, enables sharp bounds on the relative backward errors. These bounds are expressed in terms of the sequence $\|A^k\|^{1/k}$, which can be much smaller than $\|A\|$ for nonnormal matrices. 
The scaling parameter and the degrees of the Pad\'e approximants are selected to minimize the overall computational cost, which benefits from the sharp bounds and the optimal evaluation schemes for diagonal Pad\'e approximants. Furthermore, if the input matrix is (quasi-)triangular, the algorithm exploits its structure in the recovering phase. Numerical experiments demonstrate the superiority of the proposed algorithm over existing alternatives in both accuracy and efficiency. 

\end{abstract}

\section{Introduction}
The matrix exponential and its associated  \(\varphi\)-functions
\begin{equation}\label{def.phik}
	\varphi_0(A) = \e^A, \quad 
	\varphi_j(A) = 
	\frac{1}{(k-1)!}\int_{0}^{1}\e^{(1-\tau)A}\tau^{k-1}\mathrm{d}\tau
	= \sum_{k=0}^{\infty} \frac{A^k}{(k+j)!},\quad j \in\N^+,
\end{equation}   
are fundamental to exponential integrators, a class of numerical methods for the time integration of stiff systems of ordinary differential equations (ODEs) of the form~\cite{hoos10}
\begin{equation}\label{ode}
	\frac{\mathrm{d}y}{\mathrm{d}t} = F(t,y(t)) \equiv Ay(t) + g(t,y(t)),\quad y(t_0)=y_0, \quad y\in\mathbb{\C}^n, 
	\quad t\ge t_0,
\end{equation}
where $g$ is a nonlinear function and the matrix $A\in\C^{n\times n}$ typically arises from the spatial discretization of parabolic partial differential equations.
The matrix $A$ in~\eqref{ode} usually represents the Jacobian of a certain function or an approximation thereof, so it is often large and sparse.
The solution of~\eqref{ode} satisfies the variation-of-constants formula,
\begin{equation*}
	y(t) = \e^{(t-t_0)A} y_0 + \int_{t_0}^{t} \e^{(t-\tau)A}g(\tau, y(\tau)) \mathrm{d}\tau,
\end{equation*}
which, by expanding $g$ in a Taylor series at $t_0$, can be written as~\cite[Lem.~5.1]{miwr05}
\begin{equation*}\label{ode-integral-expand}
	y(t) = \e^{(t-t_0)A} y_0 + 
	\sum_{k=1}^{\infty} \varphi_k\left( (t-t_0)A \right) (t-t_0)^k g^{(k-1)}(t_0,y_0),
\end{equation*}
where $g^{(k)}$ denotes the $k$th derivative of $g$. 
A suitable truncation of this Taylor series forms the starting point of a wide range of exponential integrator methods~\cite{coma02}, \cite{hoos10}.
For example, the exponential Rosenbrock–Euler method  for~\eqref{ode} in the nonautonomous case is given by~\cite{hos09},~\cite{pope63}
\begin{equation*}
	y_{n+1}=y_n + h \varphi_1(hJ_n)F(t_n,y_n) + 
	h^2 \varphi_2(hJ_n) \frac{\partial F}{\partial t}(t_n,y_n),
	\quad J_n = \frac{\partial F}{\partial y}(t_n,y_n),
\end{equation*}
where $y_n\approx y(t_n)$, $t_n=nh$, and $h>0$ is a stepsize.
In general, different types of exponential integrator schemes can be expressed as a linear combination of the form \cite{alhi11}, \cite{grt18}, \cite{hls98}, \cite{niwr12}
\begin{equation}\label{phi.lin.comb}
	\varphi_0(A)w_0+\varphi_1(A)w_1+\dots+\varphi_p(A)w_p,
\end{equation}
where $w_j$ are certain vectors related to the approximation of the nonlinear term $g(t,y(t))$ and $p$ is related to the order of the exponential integrator.

The evaluation of~\eqref{phi.lin.comb} requires the actions of the $\varphi$-functions on vectors.
For full matrices of small to medium dimension, the direct computation of $\varphi_j(A)$ can be efficient, as these functions need to be computed only once for some integration schemes and can then be reused~\cite{hoos10}, \cite{hos09}, \cite{niwr12}.
For large scale problems, however, evaluating $\varphi_j(A)$ is often computationally infeasible, especially when the matrix $A$ is not explicitly available but is only accessible through matrix--vector products. Moreover, $\varphi_j(A)$ can be dense even when $A$ is sparse. 
Therefore, much literature has been devoted to approximating the action $\varphi_j(A)w_j$ efficiently for large $A$, avoiding the explicit computation of $\varphi_j(A)$ and its subsequent multiplication with $w_j$. 
When a priori information about the spectrum of $A$ is available, the Leja interpolation method 
can be efficient~\cite{cali07}, \cite{ckor16}, \cite{cvb04}, \cite{det23}, and as can contour exponential integration techniques~\cite{crmu23}, \cite{sctr07}, \cite{tws06}.
Truncated Taylor series expansion \cite{alhi11} is also a viable approach.
Among numerical schemes for evaluating the action, the most widely studied and effective are arguably polynomial-~\cite{holu97}, \cite{saad92} and rational~\cite{best24}, \cite{gogr14}, \cite{ragn14} Krylov subspace methods.
Owing to their demonstrated superiority in realistic problems \cite{clpu13}, \cite{etl17}, \cite{loto13}, Krylov-based exponential integrators have been implemented in a number of software packages \cite{grt18}, \cite{loto14}, \cite{niwr12}, \cite{sidj98}.

For a given vector $w_j$, the polynomial Krylov method aims to find an appropriate approximation to $\varphi_j(A)w_j$ in the Krylov subspace, say, of order $m$,
$\mathcal{K}_m(A, w_j) = 
	\spn\{
	w_j, Aw_j, \dots, A^{m-1}w_j \}$.
Initializing with $v_1 = w_j / \normt{w_j}$, an orthonormal basis $V_m= [v_1,v_2,\dots,v_m]$ of the Krylov subspace $\mathcal{K}_m(A, w_j) = 
\mathcal{K}_m(\sigma I -A, w_j)$, $\sigma\in \C$, is then built via the Arnoldi (or Lanczos) process, yielding
\[
AV_m = V_m H_m + h_{m+1,m} v_{m+1} e_m^T,
\]
where 
$H_m$ is $m\times m$ upper Hessenberg (and the upper left part of its successor $H_{m+1}$), and $e_m$ is the $m$th column of the $m\times m$ identity matrix.
By Cauchy’s integral formula~\cite[Def.~1.11]{high:FM} and the fact that $V_m(\sigma I-H_m)^{-1}e_1$ is a Galerkin approximation to $(\sigma I-A)^{-1}w_j$~\cite[Eq.~(2.3)]{holu97}, the action $\varphi_j(A)w_j$ can be projected onto a lower-dimensional Krylov subspace, leading to the approximation
\begin{equation}\label{eq:phijb-approx}
	\varphi_j(A)w_j  
	\approx \frac{\normt{w_j}}{2\pi\iu} \int_{\Gamma} \varphi_j(\sigma) V_m(\sigma I-H_m)^{-1}e_1 \mathrm{d}\sigma = 
	\normt{w_j} V_m \varphi_j(H_m)e_1,
\end{equation}
where the contour $\Gamma$ encloses the spectrum of $A$.
For the rational Krylov method, the (rational) basis vectors are constructed differently by solving shifted linear systems involving $A$, but it results in a projected approximation in the same form as~\eqref{eq:phijb-approx}~\cite{gutt13}.

The matrix $H_m$ is typically of modest size, making the efficient and accurate direct computation of $\varphi_j(H_m)$ a key step in the practical implementation of the method.
Indeed, as noted by Minchev and Wright~\cite{miwr05} and Caliari et al.~\cite{cceo24}, the main challenge in the computation of exponential integrator schemes lies in the stable and efficient evaluation of the exponential and the related $\varphi$-functions.

Despite the practical importance, the computation of $\varphi$-functions for full matrices has received less attention than that of the matrix exponential~\cite{mova78},~\cite{mova03}.
Therefore, in this paper we focus on developing a backward stable and efficient \alg\ for \emph{simultaneously} computing the $\varphi$-functions of general dense matrices of modest sizes.
The new algorithm exploits the well-established scaling and recovering idea for the $\varphi$-functions~\cite{high:FM}, \cite{laws67}.
The input matrix is initially scaled down by $2^s$ for some appropriate nonnegative integer $s$.
Given the largest index $p$ of interest, $\varphi_p$ of the scaled matrix is then approximated with a diagonal Pad\'e approximant. 
Next, nondiagonal \Pant s to $\varphi_j$, $0 \le j < p$, of the scaled matrix are obtained implicitly via a recurrence relation.
Finally, the double-argument formula~\eqref{double.formula} (see below) is invoked repeatedly to recover the effect of scaling.

This manuscript is organized as follows. 
We present in section \ref{sec.thm.frw} a theoretical framework that facilitates our analysis and method derivation. In section~\ref{sect.be.anal}, we perform a backward error analysis of the new method for computing the $\varphi$-functions by drawing a connection to the $[m+p/m]$ \Pant\ to the exponential.
The computational cost of the algorithm when using the Paterson--Stockmeyer method is analyzed in section~\ref{sect.cost.para}, enabling the choice of the optimal algorithmic parameters to minimize this cost.
In section \ref{sec.algs}, we present our new \alg, followed by a brief review of some existing \alg s.
Numerical experiments are presented in section~\ref{sect.experiment}, where the performance of the new algorithm is compared with that of the best existing algorithms.
Conclusions are drawn in section~\ref{sect.conclusion}.
\section{Theoretical framework}\label{sec.thm.frw}

We begin by presenting a general theorem that facilitates our analysis. Given an arbitrary analytic function $f$ and a structured block triangular matrix $W$, we derive a recurrence relation for the off-diagonal blocks of $f(W)$. An important corollary then follows, establishing a notable recurrence relation among the nondiagonal \Pant s to the $\varphi$-functions. We conclude the section with a basic version of the proposed algorithm.
\begin{theorem}
\label{Thm.block.phi}
Let $f$ be an analytic function on a connected open set containing the origin and the spectrum of $A\in\C^{\nbyn}$, and let $p$ be a positive integer.
Consider the block matrix
\begin{equation}
\label{block_phi}
  W = \begin{bmatrix}
         A  &  E \\ 0 & J
\end{bmatrix},
\end{equation}
where $E=\begin{bmatrix}
           I & 0 & 0 & \cdots & 0 
         \end{bmatrix}\in\R^{n\times np}
        $, $J=J_p(0)\otimes I$, and $J_p(0)$ is the Jordan block associated with zero. Then the first block row of $f(W)$ is
\begin{eqnarray*}
  [f(W)]_{1: n,1: np}&=&\begin{bmatrix}
         g_0(A) & g_1(A) & g_2(A) & \cdots  & g_p(A) 
      \end{bmatrix},
\end{eqnarray*}         
where
\begin{equation}
\label{rec.g}
  g_i(A) = Ag_{i+1}(A) + \frac{f^{(i)}(0)}{i!}I,\quad g_0=f,\quad i=0\colon p-1. 
\end{equation}
 
\end{theorem}
\begin{proof}
Since $f$ is analytic, it has a power series expansion around the origin
\[
f(x)=\sum_{k=0}^{\infty}\frac{f^{(k)}(0)}{k!}x^k.
\]
Applying this to $W$ yields
\begin{equation}
\label{f.X}
  f(W) = \begin{bmatrix}
         f(A)  &  \cL_f(A,J,E) \\ 0 & f(J_p(0))\otimes I
\end{bmatrix},
\end{equation}
where the $(1,2)$ block can be expressed as (see \cite[Lem.~1.2]{almo24},~\cite[Problem.~3.6]{high:FM})
\begin{equation*}
	\cL_f(A,J,E)
	= \sum_{k=1}^{\infty}\frac{f^{(k)}(0)}{k!}\cL_{x^k}(A,J,E) 
	=  \sum_{k=1}^{\infty}\frac{f^{(k)}(0)}{k!}\sum_{j=1}^{k}
	A^{k-j}E(J_p(0)^{j-1}\otimes I).
\end{equation*}
Defining the row vectors
\begin{equation*}
	e_j^{(p)} = \begin{cases} 
		\mbox{$j$th row of the $p\times p$ identity matrix}, 
		& \mbox{if $1\le j\le p$,} \\
		\mbox{$1\times p$ zero vector}, & \text{otherwise},
		\end{cases}
\end{equation*}
we can write $E= e_1^{(p)}\otimes I$ and hence 
$E(J_p(0)^{j-1}\otimes I) = [J_p(0)^{j-1}]_{1,1: p} \otimes I$. Then we obtain
\begin{equation*}
    \cL_f(A,J,E)
    = \sum_{k=1}^{\infty}\frac{f^{(k)}(0)}{k!}\sum_{j=1}^{k}
                 A^{k-j} (e_j^{(p)}\otimes I)
   =          \begin{bmatrix}
                g_1(A) & g_2(A) & \cdots  & g_p(A)
                \end{bmatrix},
\end{equation*}
where we have used the fact that
\begin{equation*}
	\sum_{j=1}^{k}	A^{k-j} (e_j^{(p)}\otimes I) = 
	\begin{cases} 
	[A^{k-1} \quad A^{k-2} \quad \cdots  \quad A^{k-p}],
		& \mbox{if $k \ge p$,} \\
	[A^{k-1} \quad \cdots \quad A \quad I \quad 0 \ \cdots], & \text{otherwise}.
	\end{cases}
\end{equation*}
Therefore,
\begin{equation}\label{g.series}
 g_i(A)=\sum_{k=i}^{\infty}\frac{f^{(k)}(0)}{k!}A^{k-i},\quad i=1\colon p,
\end{equation}
and the recurrence \eqref{rec.g} follows immediately by defining $g_0=f$. 
\end{proof}
The linear operator $\cL_f\colon \C^{n\times d}\mapsto  \C^{n\times d}$ is introduced and studied by Al-Mohy \cite{almo24}.
By \eqref{def.phik} and Theorem~\ref{Thm.block.phi} with $f=\exp$, we have 
\begin{equation}\label{D.exp.mat}
\cL_{\exp}(A,J,E)=\begin{bmatrix}
                \varphi_1(A) & \varphi_2(A) & \dots  & \varphi_p(A)
                \end{bmatrix}.
\end{equation}
Therefore, the problem of computing the $\varphi$-functions can be reduced to the evaluation of $\varphi_0(A)=\e^A$ and $\cL_{\exp}(A,J,E)$, which, in view of Theorem \ref{Thm.block.phi}, can be computed by extracting the first block row of the exponential of the block matrix $W$ in \eqref{block_phi}.
However, it is preferable to approximate the first block row of $\e^W$
without explicitly forming the entire matrix $\e^W$, which is of size $n(p+1) \times n(p+1)$.
For this purpose, we invoke the following corollary, which reveals an important relation among the nondiagonal \Pant s to $\varphi_j(z)$, $j=0\colon p$.
\begin{corollary}
\label{Corol.pade}
Suppose $f(z)$ in the recurrence \eqref{rec.g} is the $[m+p/m]$ \Pant, $\cR^{(0)}_{m+p,m}(z)$, to $\e^z$, whose numerator and denominator are polynomials of degrees $m+p$ and $m$, \resp. Then $g_j(z)=:\cR^{(j)}_{m+p-j,m}(z)$, $j=1\colon p$, are the $[m+p{-}j/m]$ \Pant s to $\varphi_j(z)$.
\end{corollary}
\begin{proof}
With a simple manipulation, the recurrence \eqref{rec.g} can be written as 
\[
f(z) = \sum_{i=0}^{j-1} \frac{f^{(i)}(0)}{i!} z^i + z^j g_j(z), 
\quad 1\le j\le p.
\]
For $f=\exp$, it is clear that $f^{(i)}(0)=1$ and $g_j=\varphi_j$. Now, let $f=\cR^{(0)}_{m+p,m}$. Then
$g_j(z)$ is a rational approximant to $\varphi_j(z)$ of degree $[m+p{-}j/m]$. The sufficient condition for $g_j(z)$ to be the \Pant\ to $\varphi_j(z)$ is 
\[
\left. \frac{\dt^i}{\dt z^i} \cR^{(0)}_{m+p,m}(z) \right|_{z=0} = 1, \quad 
i = 0 \colon j-1,\quad 1\le j\le p,
\]
 which follows immediately from the Pad\'e error expansion \cite[Eq.~(10.24)]{high:FM}
\begin{equation}
\label{pade.error}
 \e^z-\cR^{(0)}_{m+p,m}(z) = \frac{(-1)^m(m+p)!@m!}{(2m+p)!(2m+p+1)!}z^{2m+p+1}+O(z^{2m+p+2}).
\end{equation}
The Pad\'e error expansion
\begin{equation}\label{pade.error.j}
	\varphi_j(z) - \cR^{(j)}_{m+p-j,m}(z)=
	\frac{(-1)^m(m+p)!@m!}{(2m+p)!(2m+p+1)!}z^{2m+p-j+1}+O(z^{2m+p-j+2}),
\end{equation} 
valid for $j=0\colon p$, is obtained by first adding and subtracting the terms $\sum_{k=0}^{j-1}z^k/k!$ on the left-hand side of~\eqref{pade.error}, followed by division of both sides by $z^j$.
\end{proof}
From now on, we will denote the $[m+p{-}j/m]$ \Pant\ to $\varphi_j(z)$ simply by $\cR^{(j)}_m(z)$, as the dependency on the fixed parameter $p$ is clear.

Corollary \ref{Corol.pade} and the recurrence \eqref{rec.g} provide a practical way to evaluate the \Pant s $\cR^{(j)}_m(z)$. We begin by computing $\cR^{(p)}_m(z) := N_m(z) / D_m(z)$ using the formulae\footnote{The subscript $m$ denotes the degree; shifting it yields no valid relations among \Pant s.}
given in
\cite[sect.~5]{bsw07}, \cite[Lem.~2]{skwr09}, or \cite[Thm.~10.31]{high:FM}:
\begin{equation}
\label{pade.Nm.Dm}
\begin{aligned}
N_m(z) &= \frac{m!}{(2m + p)!} \sum_{i=0}^{m} \left[ \sum_{j=0}^{i} \frac{(2m + p - j)! (-1)^j}{j!(m - j)! (p + i - j)!} \right] z^i, \\
D_m(z) &= \frac{m!}{(2m + p)!} \sum_{i=0}^{m} \frac{(2m + p - i)!}{i!(m - i)!} (-z)^i. 
\end{aligned}
\end{equation}
We then use the recurrence
\begin{equation}\label{rec.pade}
\cR^{(j)}_m(z) = z\cR^{(j+1)}_m(z) + \frac{1}{j!},\quad j = p\!-\!1\,{:}\!-1\,{:}\,0
\end{equation} 
to recover the remaining \Pant s. This recurrence offers a computational advantage as discussed later in section \ref{sect.cost.para}.
Since \Pant s provide accurate results near the origin, it is often necessary to appropriately scale the input matrix down by a power of 2, apply the \Pant s, and then undo the effect of the scaling.
The scaling and squaring method for computing $\e^W$, where $W$ is defined in \eqref{block_phi}, exploits the relation $\e^{2W} = \left(\e^{W}\right)^2$, which is applied repeatedly to recover an approximation of $\e^W$ via an approximant to $\e^{2^{-s}W}$, for a suitably chosen nonnegative integer $s$~\cite[sect.~10.3]{high:FM}.
Equating the first block row of this relation yields the double-argument formula
\cite{bsw07}, \cite{skwr09}
\begin{equation}\label{double.formula}
\varphi_j(2A) = \frac{1}{2^j} \left( \varphi_0(A)\varphi_j(A) +
\sum_{k=1}^{j} \frac{\varphi_k(A)}{(j-k)!} \right), \quad j = p\colon-1\colon 0,
\end{equation}
which is used recursively to recover $\varphi_j(A)$ from $\varphi_j(2^{-s}A)$, $j=0:p$.
\begin{algorithm}[t]
\caption{Basic version \alg\ for computing $\varphi$-functions.}
\label{alg.basic}
\begin{algorithmic}[1]
    \State \textbf{Inputs:} $A\in\C^{\nbyn}$ and $p\in\N^+$. 
    \State \textbf{Outputs:} Approximations for $\varphi_0(A)=\e^A$, $\varphi_1(A)$, $\varphi_2(A)$, \dots, $\varphi_p(A)$. 
    \State Select a scaling parameter $s$ and a degree $m$ of \Pant.
    \State $A \gets A / 2^s$     
    \State Evaluate the $[m/m]$ \Pant, $\cR^{(p)}_m\equiv \cR^{(p)}_m(A)$, to $\varphi_p(A)$ using
     \eqref{pade.Nm.Dm}.\label{lin.cost.pade}     
    \State Invoke the recurrence \eqref{rec.pade} to recover the approximants $\cR^{(j)}_m\equiv \cR^{(j)}_m(A)$ to $\varphi_j(A)$.\label{lin.phase1}
    \For{$i = 1 \colon s$} \label{lin.phase2.on}
        \For{$j = p\colon\!-1\,\colon0$}
            \State $\cR^{(j)}_m \gets 2^{-j} \left( \cR^{(0)}_m \cR^{(j)}_m + \sum_{k=1}^{j} \cR^{(k)}_m / (j-k)! \right)$ 
            \Comment{using \eqref{double.formula}} 
        \EndFor
    \EndFor \label{lin.phase2.off}
\end{algorithmic}
\end{algorithm}
The basic computational framework for the $\varphi$-functions is presented as Algorithm~\ref{alg.basic}, and the key task is to determine the scaling parameter $s$
and the degree of \Pant\ in such a way that the overall computational cost is minimized.
\section{Backward error analysis} 
\label{sect.be.anal}
In this section, we analyze the backward error arising from approximating the $\varphi$-functions using the scaling and recovering method with \Pant s.
We relate the backward error of the approximation of the $\varphi$-functions to that of the matrix exponential by employing the $[m+p/m]$ \Pant\ to $\e^z$. We recall two important backward error analyses: the work of Al-Mohy and Higham \cite[sect.~3]{alhi09a} and that of Al-Mohy \cite[sect.~3]{almo24}. In view of the recurrence \eqref{rec.pade} and \eqref{pade.Nm.Dm}, all the \Pant s \emph{share} the same denominator polynomial $D_m$.
\begin{theorem}\label{thm.berr}
Let $\cR^{(0)}_m(z)$ be the $[m+p/m]$ \Pant\ to $\e^z$ whose denominator polynomial is $D_m(z)$ and
\begin{equation}\label{Omega-m}
  \Omega^{(n)}_{m,p} = \{\,Z\in\C^{\nbyn} : \rho( \e^{-Z} \cR^{(0)}_m(Z) - I ) < 1,\quad
  \rho(Z) < \nu_{m,p} \,\},
\end{equation}
where $\nu_{m,p} = \min\{\,|z| : D_m(z) = 0 \,\}$ and $\rho$ denotes the spectral radius. Then the function
\begin{equation}
      h_{m,p}(X) = \log( \e^{-X} @\cR^{(0)}_m(X) )
      \label{h2mp1}
\end{equation}
is defined for all $X\in\Omega^{(n)}_{m,p}$, where $\log$ is the
principal matrix logarithm, and 
\[
\cR^{(0)}_m(X)= \e^{X+h_{m,p}(X)},\quad X\in\Omega^{(n)}_{m,p}.
\]
\end{theorem}
If $X$ is not in $\Omega^{(n)}_{m,p}$, choose $s$ so that $2^{-s}X \in
\Omega^{(n)}_{m,p}$. Then
\begin{equation}
\left[\cR^{(0)}_m(2^{-s}X)\right]^{2^s} = \e^{X+ 2^s h_{m,p}(2^{-s}X)} =: \e^{X+\D X}
\label{be.eq}
\end{equation}
and the matrix $\D X = 2^s h_{m,p}(2^{-s}X)$ represents the backward error resulting from the approximation of $\e^X$ by the scaling and squaring method via \Pant s.
Over $\Omega^{(n)}_{m,p}$, the function $h_{m,p}$ has the power series
expansion
\begin{eqnarray}
h_{m,p}(X) &=& \sum_{k=2m+p+1}^\infty c_{m,p,k}@@X^k \nonumber\\
     &=& \frac{(-1)^{m+1}(m+p)!@m!}{(2m+p)!(2m+p+1)!}@@X^{2m+p+1}+O(X^{2m+p+2}).
     \label{power.h2m}
\end{eqnarray}

Applying the $\cL$ operator associated with the mappings defined in \eqref{be.eq} to the ordered triplet $(A, J, E)$ from Theorem~\ref{Thm.block.phi}, we obtain~\cite[Eq.~(3.5)]{almo24}
\begin{eqnarray}
\cL_{\left[\cR^{(0)}_m(2^{-s}x)\right]^{2^s}}(A,J,E)=
  \cL_{\exp}(A+\D A,J+\D J,E+\D E),
  \label{apply.D}
\end{eqnarray}
where $\D A \!=\! 2^s h_{m,p}(2^{-s}\!A)$, $\D J\! =\! 2^sh_{m,p}(2^{-s}\!J)$, and
$\D E\! = \cL_{h_{m,p}}(2^{-s}\!A,2^{-s}J,E)$ are the \emph{overall} backward errors with respect to the original matrices $A$, $J$, and $E$ resulting from the approximation via the scaling and squaring method with the $[m+p/m]$ \Pant.
Since $J$ is nilpotent of index $p$, we have $\D J = 0$. This is expected as \eqref{pade.error} indicates that $\cR^{(0)}_m(J)=\e^J$. Therefore, we need to control the backward errors $\D A$ and $\D E$. 
Using the $1$-norm,
we have
\begin{equation}
 \label{berr-bound1}
  \frac{ \normi{\D A} }{ \normi{A} } = \frac{2^s \normi{ h_{m,p}(2^{-s}A)} }
                                              { \normi{A} }             
                                   = \frac{\normi{ h_{m,p}(2^{-s}A)} }
                                              { \normi{2^{-s}A} }\\
                                  \le \frac{\Ht_{m,p}(\th)}{\th},           
\end{equation}
where $\Ht_{m,p}(\th) = \sum_{k=2m+p+1}^\infty |c_{m,p,k}|@@\th^k$ is obtained from \eqref{power.h2m} and $\th$ is a nonnegative real-valued function of $2^{-s}A$ (the singularity of $\Ht_{m,p}(\th)/\th$ is removable). A simple choice for $\th$ is $\normi{2^{-s}A}$; however, 
we show that choosing $\th$ differently yields a tighter bound.
Define \cite[sect.~3]{alhi09a}
\begin{equation}
 \th_{m,p} = \max\{\,\th : \Ht_{m,p}(\th)/\th \le u \,\},
  \label{th-mp}
\end{equation}
 where $u  = 2^{-53} \approx 1.1\times
10^{-16}$ is the unit roundoff for IEEE double precision
arithmetic. We compute this $\th_{m,p}$ for $m=1\colon 20$ and $p=1\colon 10$ and list selected values in Table \ref{tab:theta}, where some values of $\th_{m,p}$ are adjusted based on our cost and \be\ analysis below. 
\begin{table}
	\caption{Selected values of $\theta_{m_i,p}$ as defined in~\eqref{th-mp}, refined using~\eqref{berr-bound2}. Each $\theta_{m_i,p}$ lies in the disc $\{z \in \C : |z| < \nu_{m_i,1}\}$, where $m_i$ is given by~\eqref{m.opt.orders}.}
	\label{tab:theta}
	\centering
\scalebox{.85}{
\begin{tabularx}{\textwidth+0.6cm}{@{\extracolsep{\fill}}c@{\hspace{.6cm}}cccccccc}
\toprule[1pt] 
$p$ & $m_0=1$ & $m_1=2$ & $m_2=3$ & $m_3=4$ & $m_4=6$ & $m_5=8$ & $m_6=10$ & $m_7=12$  \\
\midrule 
 1  & 2.00e-5 & 3.81e-3 & 3.97e-2 & 1.54e-1 & 7.26e-1 & 1.76  & 3.17  & 4.87    \\
 2  & 3.76e-5 & 6.09e-3 & 5.81e-2 & 2.13e-1 & 9.28e-1 & 2.06  & 3.54  & 5.28    \\
 3  & 7.37e-5 & 9.87e-3 & 8.53e-2 & 2.94e-1 & 1.16  & 2.37  & 3.91  & 5.69    \\
 4  & 1.50e-4 & 1.62e-2 & 1.26e-1 & 4.06e-1 & 1.40  & 2.69  & 4.28  & 6.09    \\
 5  & 3.15e-4 & 2.70e-2 & 1.87e-1 & 5.62e-1 & 1.66  & 3.01  & 4.65  & 6.50    \\
 6  & 6.86e-4 & 4.55e-2 & 2.80e-1 & 7.79e-1 & 1.92  & 3.34  & 5.02  & 6.90    \\
 7  & 1.54e-3 & 7.75e-2 & 4.18e-1 & 1.05  & 2.20  & 3.68  & 5.40  & 7.30    \\
 8  & 3.54e-3 & 1.33e-1 & 6.26e-1 & 1.26  & 2.48  & 4.01  & 5.77  & 7.69    \\
 9  & 8.35e-3 & 2.30e-1 & 9.34e-1 & 1.48  & 2.77  & 4.35  & 6.14  & 8.08    \\
10  & 2.01e-2 & 3.99e-1 & 1.16  & 1.71  & 3.07  & 4.69  & 6.51  & 8.47    \\
\midrule 
$\nu_{m_i,1}$  & 3.00  & 4.47  & 5.65   & 7.05   &  9.68   & 1.23e1 & 1.50e1 & 1.76e1   \\
\bottomrule[1pt] 
\end{tabularx} 
}
\end{table}
Before we proceed, it is essential to verify that the evaluation of the $[m/m]$ \Pant, $\cR^{(p)}_m(z) = N_m(z) / D_m(z)$, and the recurrence \eqref{rec.pade} are well defined for all $z$ lying in the disc centered at the origin with radius $\th_{m,p}$.
This has been confirmed symbolically.
We compute $\nu_{m,p}$, the smallest modulus of the poles of $\cR^{(p)}_m(z)$ as defined in Theorem \ref{thm.berr}, and observe that, 
with either index fixed, $\nu_{m,p}$ increases as the other index increases.
Moreover, $\th_{m,p} < \nu_{m,1} \le \nu_{m,p}$ for $m\le 20$ and $p\le 10$.
The last row of Table \ref{tab:theta} lists selected values of $\nu_{m,1}$.
For fixed $m$ and $p$, the function $\th^{-1}\Ht_{m,p}(\th)$ is increasing over $(0,\infty)$. Thus, $\th_{m,p}$ is the unique point at which equality in \eqref{th-mp} is attained.
However, for a fixed $\th$, the sequence $\th^{-1}\Ht_{m,p}(\th)$ decreases as either $m$ or $p$ increases.  
Thus, the parameters $\th_{m,p}$ form an increasing sequence in either $m$ or $p$. That is, $\th_{m,1}<\th_{m,2}<\cdots<\th_{m,k}<\cdots$ and
$\th_{1,p}<\th_{2,p}<\cdots<\th_{k,p}<\cdots$.

For the \be\ $\D E$, Al-Mohy \cite[Thm.~3.2]{almo24} derives a bound for the relative backward error $\norm{\D E}/\norm{E}$ for any subordinate matrix norm. However, the special structure of the problem here allows us to derive a
sharper bound. 
Using \eqref{g.series} with $g_0=h_{m,p}$, the $j$th block of $\D E$ is
\begin{eqnarray*}
  (\D E)_j &=& \sum_{k=2m+p+1}^{\infty}c_{m,p,k}@@(2^{-s}A)^{k-j}\\
  &=& (2^{-s}A)^{-j} h_{m,p}(2^{-s}A),
  \quad 1\le j\le p,
\end{eqnarray*}
where the singularity is removable.
Therefore, 
$\normi{(\D E)_j}\le\th^{-j}\Ht_{m,p}(\th)$ for a suitably chosen $\th$ and
\begin{eqnarray}
  \frac{\normi{\D E}}{\normi{E}}=\normi{\D E}
           &\le& \max\left\{\,\th^{-1}\Ht_{m,p}(\th),\,\th^{-2}\Ht_{m,p}(\th),\cdots,\,
                  \th^{-p}\Ht_{m,p}(\th)\,\right\}\nonumber\\
           &=&\begin{cases}                   \label{berr-bound2}
  \th^{-1}\Ht_{m,p}(\th),         & \quad \th\ge 1, \\
  \th^{-p}\Ht_{m,p}(\th), &  \quad 0 < \th < 1,
\end{cases}         
\end{eqnarray}
recalling that $\normi{E}=1$. 
In view of this bound, and together with \eqref{berr-bound1} and \eqref{th-mp}, we have $\normi{\D A}\le u\normi{A}$ and $\normi{\D E}\le u$ if $1\le\th\le\th_{m,p}$. However, $\normi{\D E}$ can exceed $u$ if $0<\th\le\th_{m,p}<1$ since $\th^{-p}\Ht_{m,p}(\th) > u$ for $\th=\th_{m,p}$.
To address this issue, those $\th_{m,p}<1$ of Table \ref{tab:theta} are \emph{recomputed} after replacing 
$\th^{-1}\Ht_{m,p}(\th)$ with $\th^{-p}\Ht_{m,p}(\th)$ in \eqref{th-mp}.
Consequently, those recomputed values are smaller than the original ones, ensuring that $\normi{\D A}\le u\normi{A}$ remains valid. They also preserve the monotonicity property of the sequence $\th_{m,p}$. For $m\ge7$ and any $p$, the original values of $\th_{m,p}$ remain unchanged since they are already greater than one.  

Now, if we choose $s$ such that $\th=\normi{2^{-s}A}\le\th_{m,p}$, a straightforward choice corresponding to the classical approach, then the backward error bounds in \eqref{berr-bound1} and \eqref{berr-bound2} will not exceed $u$ in exact arithmetic. 
However, Al-Mohy and Higham \cite[Alg.~6.1]{alhi09a} propose a more liberal choice of the scaling parameter for the matrix exponential, currently implemented in the MATLAB function \expm. A key objective of their algorithm was to overcome the overscaling phenomenon inherent in the classical scaling and squaring method and to reduce computational cost.
Accordingly, we bound the relative backward errors in \eqref{berr-bound1} and \eqref{berr-bound2} by choosing $s$ and $\th=\a_r(2^{-s}A)$, where 
\begin{equation}
  \a_r(A) = \max\left( \normi{A^r}^{1/r},\normi{A^{r+1}}^{1/(r+1)}
              \right),
             \label{def-alphap}
\end{equation}
and $r$ is selected to minimize $\a_r(A)$ subject to the constraint $2m+\phat+1 \ge r@(r-1)$ \cite[Eq.~(5.1)]{alhi09a} with
\[
\phat = \begin{cases}
 	 p, &  \text{if } \th_{m,p}\ge1, \\
 	 0, &  \text{otherwise}.
 \end{cases}
\]
This choice caters to the bound in \eqref{berr-bound2}. The largest value of the positive integer $r$ satisfying the constraint is
\begin{equation}
\label{r_m}
r_{\max} = \left\lfloor \frac{1 + \sqrt{1 + 4(2m + \phat + 1)}}{2} \right\rfloor.
\end{equation}
The advantage of using $\a_r(A)$ rather than $\normi{A}$ is that $\a_r(A)$ can be significantly smaller and closer to the spectral radius of $A$ than $\normi{A}$, especially for highly nonnormal matrices. Thus, by \cite[Thm.~4.2(a)]{alhi09a} we have 
\begin{equation*}
\normi{h_{m,p}(2^{-s}A)}\le\Ht_{m,p}\bigl(\a_r(2^{-s}A)\bigr)
\end{equation*}
and, consequently, the relative backward errors in \eqref{berr-bound1} and \eqref{berr-bound2} are confined by the bound: 
\begin{equation}
\label{berr.alphar.bound}
\max\left(\frac{\normi{\D A}}{\normi{A}},\frac{\normi{\D E}}{\normi{E}}\right)\le
\frac{\Ht_{m,p}\bigl(\a_r(2^{-s}A)\bigr)}
{\a_r(2^{-s}A)^\delta},
\quad \delta = 
(p-1)(p-\phat\,)p^{-1}+1.
\end{equation}
Thus, \Alg\ \ref{alg.basic} with the selection of $s$ and $m$ such that $2^{-s}\a_r(A)\le\th_{m,p}$ for any $2\le r\le r_{\max}$, guarantees that the \be\ bound \eqref{berr.alphar.bound} does not exceed $u$ in exact arithmetic. However, since $\a_r(A) \ll \normi{A}$ can occur for some matrices, the polynomials $N_m$ and $D_m$ in \eqref{pade.Nm.Dm} evaluated at $2^{-s}A$ may not be sufficiently accurate. Al-Mohy and Higham \cite[sect.~5]{alhi09a} address this issue for the matrix exponential by proposing a modification to the scaling parameter selection, which performs satisfactorily in practice. 
We take a similar approach by employing the first term in the \be\ series  \eqref{power.h2m} using the matrix $|A|$.
Suppose $t$ is a nonnegative integer such that $2^{-t}A\in\Omega^{(n)}_{m,p}$, with $\Omega^{(n)}_{m,p}$ as defined in \eqref{Omega-m}. 
Using the fact that $|h_{m,p}(2^{-t}A)|\le \Ht_{m,p}(2^{-t}|A|)$ and returning to the bound \eqref{berr.alphar.bound}, we have
\begin{eqnarray*}
\max\left(\frac{\normi{\D A}}{\normi{A}},\frac{\normi{\D E}}{\normi{E}}\right)
  &\le& \frac{\normi{ \Ht_{m,p}(2^{-t}|A|) } }{ \normi{2^{-t}A}^\delta }\\
  &=&  \frac{(m+p)!@m!}{(2m+p)!(2m+p+1)!}@@\frac{ \normi{@|2^{-t}A|^{2m+p+1}} }{\normi{2^{-t}A}^\delta}+\cdots.  
\end{eqnarray*} 
Thus, we choose $t$ so that the first term of the right-hand side does not exceed
$u$. That is,
\begin{equation}\label{scal.parm.t}
 t = \max\left( \left\lceil \log_2\left(  \frac{(m+p)!@m!}{(2m+p)!(2m+p+1)!} \frac{ \normi{@|A|^{2m+p+1}} }{ u \normi{A}^\delta  } \right) / (2m+p+1{-}\delta) \right\rceil , 0 \right).
\end{equation}
We now take the scaling parameter 
\begin{equation}\label{s.star}
s^* = \max(s, t),
\end{equation}
 for which it is evident that the \be\ bound in \eqref{berr.alphar.bound} does not exceed $u$ in exact arithmetic.

The matrix powers in the sequence $\normi{A^r}$ given in \eqref{def-alphap} are not computed explicitly before taking the 1-norm. Instead, $\normi{A^r}$ is estimated through several actions of $A^r$ on specific vectors using the block 1-norm estimation algorithm of Higham and Tisseur \cite{hiti00n}, which requires only $O(n^2)$ operations. Furthermore, quantities of the form $\normi{|A|^k}$, as in \eqref{scal.parm.t}, can be
computed exactly in $O(n^2)$ operations exploiting the identity
$
\normi{@|A|^k} = \normo{@|A^T|^ke},
$
where $e=[1,1,\cdots,1]^T$ \cite[sect.~5]{alhi09a}.

\begin{table}
	\caption{Upper bounds on $\kappa_A(D_{m_i}(A))$ corresponding to the parameter settings in Table~\ref{tab:theta}.}
	\label{tab:bound}
	\centering
\scalebox{.85}{
\begin{tabularx}{\textwidth+0.2cm}{@{\extracolsep{\fill}}c@{\hspace{.6cm}}cccccccc}
\toprule[1pt]
$p$ & $m_0=1$ & $m_1=2$ & $m_2=3$ & $m_3=4$ & $m_4=6$ & $m_5=8$ & $m_6=10$ &$m_7=12$  \\
\midrule 
 1  & 1.00  & 1.00  & 1.03  & 1.13  & 1.86  & 4.78  & 1.79e1 & 8.98e1   \\
 2  & 1.00  & 1.00  & 1.04  & 1.17  & 2.10  & 5.68  & 2.18e1 & 1.10e2   \\
 3  & 1.00  & 1.00  & 1.05  & 1.22  & 2.39  & 6.68  & 2.61e1 & 1.32e2   \\
 4  & 1.00  & 1.01  & 1.07  & 1.28  & 2.69  & 7.78  & 3.08e1 & 1.56e2   \\
 5  & 1.00  & 1.01  & 1.10  & 1.38  & 3.02  & 8.98  & 3.60e1 & 1.83e2   \\
 6  & 1.00  & 1.02  & 1.14  & 1.52  & 3.38  & 1.03e1 & 4.17e1 & 2.11e2   \\
 7  & 1.00  & 1.03  & 1.20  & 1.69  & 3.75  & 1.17e1 & 4.77e1 & 2.42e2   \\
 8  & 1.00  & 1.04  & 1.28  & 1.81  & 4.14  & 1.31e1 & 5.40e1 & 2.74e2   \\
 9  & 1.00  & 1.07  & 1.42  & 1.93  & 4.56  & 1.47e1 & 6.07e1 & 3.07e2   \\
10  & 1.00  & 1.11  & 1.51  & 2.06  & 4.98  & 1.62e1 & 6.76e1 & 3.42e2   \\
\bottomrule[1pt]
\end{tabularx} 
              }
\end{table}

The other potential concern is the quality of the computed solution $R^{(p)}_m$ to the multiple right-hand side
linear system $D_m(A)R^{(p)}_m = N_m(A)$, assuming that $A\in\Omega^{(n)}_{m,p}$ and $\a_r(A)\le\th_{m,p}$. Since $\rho(A)\le\a_r(A)$ always holds and $\th_{m,p}<\nu_{m,p}$ applies for the values of $m$ and $p$ of interest, the coefficient matrix $D_m(A)$ is nonsingular as all its eigenvalues lie within the disc centered at the origin with radius $\th_{m,p}$. Moreover, the function $D_{m}(z)^{-1}$ has an absolutely convergent power series expansion 
$D_{m}(z)^{-1}=\sum_{i=0}^{\infty}b_i z^i$ inside this disc.
We aim to bound the condition number of the coefficient matrix with respect to a suitable matrix norm. This will help to determine an appropriate range for selecting the values of $\th_{m,p}$. Following the analysis of Al-Mohy and Higham \cite[sect.~5]{alhi09a}, for a given $\epsilon > 0$ there exists a consistent matrix norm associated with the matrix $A$, denoted by $\|\cdot \|_A$  such that
\begin{equation*}
\| A \|_A \leq \rho(A) + \epsilon \leq \alpha_r(A) + \epsilon.
\end{equation*}
Hence, the corresponding condition number satisfies
\begin{eqnarray}
  \kappa_A(D_{m}(A)) &=& \| D_{m}(A) \|_A \| D_{m}(A)^{-1} \|_A \nonumber\\
    &\leq& \left(
    \frac{m!}{(2m+p)!}
    \sum_{i=0}^{m} \frac{(2m + p - i)!}{i!(m - i)!} (\alpha_r(A) + \epsilon)^i
    \right) 
    \sum_{i=0}^{\infty} |b_i| (\alpha_r(A) + \epsilon)^i \nonumber  \\
    &\leq& \left(
    \frac{m!}{(2m+p)!}
    \sum_{i=0}^{m} \frac{(2m + p - i)!}{i!(m - i)!} (\th_{m,p} + \epsilon)^i
    \right) 
    \sum_{i=0}^{\infty} |b_i| (\th_{m,p} + \epsilon)^i. 
    \label{cond.bound} 
\end{eqnarray}
We choose $\epsilon=u$ and evaluate the bound \eqref{cond.bound} for various values of $m$ and $p$ as presented in Table~\ref{tab:bound}, corresponding to the parameter settings in Table~\ref{tab:theta}. 
\section{Computational cost analysis and parameter selection}
\label{sect.cost.para}
In the previous section, we described how to determine the optimal scaling parameter $s^*$.
Here, we focus on selecting the optimal degrees for the $[m/m]$ diagonal \Pant\ $\cR^{(p)}_m$ to $\varphi_p$, and we analyze the overall computational cost of the algorithm in terms of the equivalent number of matrix multiplications. The algorithmic parameters are then chosen to minimize this cost.

First, we use the Peterson--Stockmeyer scheme~\cite{past73} in line \ref{lin.cost.pade} of \Alg~\ref{alg.basic} to evaluate the numerator- and denominator polynomials of $\cR^{(p)}_m(A)=D_m(A)^{-1}N_m(A)$. We have
\begin{equation}\label{ps.scheme}
q(A) = \sum_{k=0}^{\nu} B^{[q]}_k(A)\cdot (A^\tau)^k, \quad 
\nu = \left\lfloor \frac{m}{\tau} \right\rfloor,
\end{equation}
where the polynomial $q$ denotes either $N_m$ or $D_m$ and
\[
B_k^{[q]}(A) = 
\begin{cases}
\displaystyle \sum_{j=0}^{\tau-1} c^{[q]}_{\tau k+j} A^j, & k = 0, 1, \dots, \nu - 1, \\[10pt]
\displaystyle \sum_{j=0}^{m-\tau\nu} c^{[q]}_{\tau \nu+j} A^j, & k = \nu.
\end{cases}
\]
Fasi \cite[sect.~3]{fasi19} analyzes the number of matrix multiplications required to simultaneously evaluate $N_m$ and $D_m$, which is:
\begin{equation}\label{Cm,pade}
 \pi_m(\tau)= \tau-1 + 2\left(\left\lfloor \frac{m}{\tau} \right\rfloor - \delta_{m,\tau}\right), \quad
 \delta_{m,\tau} = \begin{cases}
 	 1, &  \text{if } m \mid \tau, \\
 	 0, &  \text{otherwise}.
 \end{cases}
\end{equation}
He also shows that the minimum is attained at $\tau_*$ dividing $m$, where $\tau_*=\left\lfloor\sqrt{2m}\right\rfloor$ or $\tau_*=\left\lceil\sqrt{2m}\,\right\rceil$ \cite[Lem.~2]{fasi19}. 
The sequence $\pi_m(\tau_*)$ is nondecreasing, so we are interested in the largest $m$ values among those with the same $\pi_m(\tau_*)$. These values represent the optimal degrees of the evaluation scheme \eqref{ps.scheme} and are
given by the sequence \cite[Eq.~(19)]{fasi19}
\begin{equation}\label{m.opt.orders}
m_i:=\left\lfloor \frac{(i+3)^2}{8} \right\rfloor,\quad i=0,\,1,\,2,\cdots,
\end{equation}
and, therefore, $\pi_{m_{i}}(\tau_*) = i$. Solving \eqref{m.opt.orders} for $i$ gives
\begin{equation}\label{mi.pi}
 \left\lceil \sqrt{8(m_i+1)} - 3\right\rceil = i+1 = \pi_{m_{i}}(\tau_*)+1.
\end{equation}

Second, the solution of the multiple right-hand side linear system for the \Pant\ requires $8n^3/3$ flops, which is equivalent in operation count to $4/3$ matrix multiplications. Third, the first recovering phase in line~\ref{lin.phase1} of \Alg~\ref{alg.basic} that uses the recurrence \eqref{rec.pade} requires $p$ matrix multiplications. Finally, the last recovering phase between line \ref{lin.phase2.on} and line \ref{lin.phase2.off} that invokes the recurrence \eqref{double.formula} requires 
$s^*(p+1)$ matrix multiplications, where $s^* = \max(s, t)$ is defined in \eqref{s.star}. Therefore, the total cost using the optimal degrees \eqref{m.opt.orders} is equivalently
\begin{equation}
\label{cost.total1}
 C_{m_i,s^*} = i + p + \frac{4}{3} + s^*(p+1)
\end{equation}
matrix multiplications.
Since the smallest value of the nonnegative scaling parameter $s$ such that $2^{-s}\a_r(A)\le\th_{m,p}$ is 
$s = \max \bigl( \left\lceil \log_2 \left( \a_r(A)/\th_{m,p} \right) \right\rceil, 0 \bigr)$, the cost function becomes
\begin{equation}
\label{cost.total2}
 C_{m_i,r} = i + p + \frac{4}{3} + \max \left( \left\lceil \log_2 \left( \a_r(A)/\th_{m,p} \right) \right\rceil, t \right)(p+1),
\end{equation}
which, for a given matrix $A$, depends on the degree $m_i$ and the parameter $r$ from~\eqref{def-alphap}.
Our goal is to minimize the cost function over the index pair $(m_i, r)$. First, we set an upper bound for $m_i$, denoted by $m_{\max}$, which limits $r$ to the range between 2 and $r_{\max}$ as defined in \eqref{r_m}. Next, we determine $i_{\max}$ associated with $m_{\max}$ from \eqref{mi.pi}. 
If $m_{\max}$ is not initially among the $m_i$ values, we use \eqref{m.opt.orders} to adjust it to the nearest smaller value, $m_{i_{\max}}$, ensuring optimality of the evaluation scheme.
Finally we seek a pair $(i^*,r^*)$ that minimizes $C_{m_i,r}$ for $i=0\colon i_{\max}$ and $r=2\colon r_{\max}$, with each $r$ constrained by $2m_i+\phat+1\ge r(r-1)$. 

Based on the condition number bounds for the denominator polynomial $D_m(A)$ listed in Table \ref{tab:bound}, we recommend setting $m_{\max}$ to 12 and $\th_{m,p}=\th_{m,7}$ for $p>7$ to keep the condition number reasonably small. This adjustment does not compromise the \be\ bound \eqref{berr-bound2}, thanks to the fact that $\th_{m,p}>\th_{m,7}$ for all $p>7$.
\section{Proposed and existing algorithms}
\label{sec.algs}
To set the stage for comparison, we start with presenting  the newly proposed \alg, followed by a brief comparative review of existing \alg s.
\subsection{The new algorithm}
Given a matrix $A \in \C^{n \times n}$ and a positive integer $p$, \Alg~\ref{alg.phi} simultaneously evaluates the matrix $\varphi$-functions $\varphi_j$ for $j = 0\colon p$. It serves as a comprehensive implementation based on the analysis developed in the previous sections.

\begin{algorithm}
\caption{Scaling and recovering algorithm for the matrix $\varphi$-functions.}
\label{alg.phi}
\textbf{Inputs:} Matrix $A \in \mathbb{C}^{n \times n}$ and $p\in\N^+$. \\
\textbf{Outputs:} Approximations for $\varphi_0(A)=\e^A$, $\varphi_1(A)$, $\varphi_2(A)$, \dots, $\varphi_p(A)$.
\begin{algorithmic}[1]
    \State $m_{\max} = 12$ \Comment{Default} 
    \State $i_{\max} = \left\lceil \sqrt{8(m_{\max}+1)} - 3 \right\rceil - 1$
    \Comment{See \eqref{mi.pi}}
    \State $m_{\max} \gets \left\lfloor (i_{\max} + 3)^2 / 8 \right\rfloor$ \Comment{Adjust to the nearest if $m_{\max}$ is not the default.}
    \If{$p > 7$}
        \State $\th_{m,p} \gets \th_{m,7}$ for all $m$
    \EndIf
    \State $\phat = p$
    \If{$\th_{m_{\max},p} < 1$}
        \State $\phat \gets 0$
    \EndIf
    \State $r_{\max} = \left\lfloor (1 + \sqrt{5 + 8m_{\max} + 4\phat}\,)/2    
    \right\rfloor$  
     \Comment{See \eqref{r_m}}
    \State Evaluate $\a_r(A)$ for $r = 2 : r_{\max}$ using the 1-norm estimator \cite{hiti00n}. 
    \State Initialize $M$ to be an $(i_{\max} + 1) \times (r_{\max} - 1)$ zero matrix.  
    \For{$i = 0 \colon i_{\max}$}
        \State $m_i = \left\lfloor (i + 3)^2 / 8 \right\rfloor$, $\phat \gets p$
        \If{$\th_{m_i, p} < 1$}
            \State $\phat \gets 0$
        \EndIf
        \State Evaluate the scaling parameter $t$ for $m_i$, $p$, and $\phat$, using \eqref{scal.parm.t}.
        \For{$r = 2 \colon r_{\max}$} \label{alg.phi-line.r2}
            \If{$2m_i + \phat + 1 \geq r(r - 1)$} 
                \State $M(i + 1, r-1) = C_{m_i, r}$ \Comment{Cost function \eqref{cost.total2}}
            \EndIf
        \EndFor
    \EndFor
    \State Index the smallest positive element in $M$, denoting it $M(i^* + 1, r^*-1)$.
    \State $m = m_{i^*}$, $\tau_* = \left\lfloor \sqrt{2m} \right\rfloor$
    \If{$\pi_m(\tau_*) \ne i^*$}
        \State $\tau_* \gets \left\lceil \sqrt{2m} \right\rceil$ \Comment{See \eqref{Cm,pade}}
    \EndIf
    \State $s^* =  \left(M(i^* + 1, r^*-1)-i^*-p-4/3\right)/(p+1)$ \Comment{See \eqref{cost.total1}}
    \State $A \gets A / 2^{s^*}$
    \State Evaluate $\cR^{(p)}_m\equiv \cR^{(p)}_m(A)$ to $\varphi_p(A)$ using \eqref{pade.Nm.Dm} and \eqref{ps.scheme}, with $\tau = \tau_*$.
    \State Invoke the recurrence \eqref{rec.pade} to recover $\cR^{(j)}_m\equiv \cR^{(j)}_m(A)\approx\varphi_j(A)$.
    \For{$i = 1 \colon s^*$}
        \For{$j = p\colon\!-1\,\colon1$}
            \State $\cR^{(j)}_m \gets 2^{-j} \left( \cR^{(0)}_m \cR^{(j)}_m + \sum_{k=1}^{j} \cR^{(k)}_m / (j-k)! \right)$ 
            \Comment{See \eqref{double.formula}} 
        \EndFor
        \If{$A$ is (quasi-)triangular}
            \State Invoke \cite[Code Fragments.~2.1 \& 2.2]{alhi09a} for $\cR^{(0)}_m$.
        \Else
            \State $\cR^{(0)}_m \gets \left(\cR^{(0)}_m\right)^2$
            \Comment{Repeated squaring}
        \EndIf
    \EndFor
\end{algorithmic}
\end{algorithm}

The \alg\ begins by determining the optimal Pad\'e degree $m$ and scaling parameter $s^*$ through the minimization of the cost function $C_{m_i,r}$ in~\eqref{cost.total2}, which measures the total cost in the equivalent number of matrix multiplications. This cost function involves the sequence $\a_r(A)$ and the parameter $t$ from~\eqref{scal.parm.t}, both of which rely on the \be\ analysis. The algorithm then evaluates the $[m/m]$ \Pant\ to $\varphi_p(2^{-s^*}A)$ using~\eqref{pade.Nm.Dm} and~\eqref{ps.scheme} with $\tau = \tau_*$, and it obtains the $[m+p{-}j/m]$ \Pant s to $\varphi_j(2^{-s^*}A)$, for $0 \le j < p$, implicitly via the backward recurrence~\eqref{rec.pade}. The effect of scaling is subsequently reversed by using the double-argument formula~\eqref{double.formula}. When the input matrix is (quasi-)triangular, the algorithm exploits this structure in the recovering phase to enhance stability. 

The method is optimized for IEEE double-precision arithmetic and balances accuracy and cost through adaptive parameter selection and structure-aware computation.
\subsection{Existing algorithms}
Inspired by the scaling and squaring method for the matrix exponential \cite{mova78}, the scaling and recovering method for the $\varphi$-functions was first proposed, to the best of our knowledge, by Hochbruck, Lubich, and Selhofer \cite{hls98}. In their approach, the authors suggest scaling the matrix by a power of two so that the norm of the scaled matrix is less than $1/2$, using the $[6/6]$ \Pant\ to $\varphi_1(z)$, and applying the double-argument formula \eqref{double.formula} for $p = 1$ to compute $\varphi_1(A)$. Neither the choice of the scaling parameter nor the degree of the \Pant\ is justified in this preliminary investigation.

The most notable works aligned with ours in \Alg~\ref{alg.phi} are those by Berland, Skaflestad, and Wright~\cite{bsw07}, and by Skaflestad and Wright~\cite{skwr09}. The latter extends the former by providing a more comprehensive description of the method, along with a forward error analysis and a detailed cost analysis.
The algorithm of Skaflestad and Wright~\cite[Alg.~1]{skwr09} simultaneously evaluates the matrix $\varphi$-functions\footnote{The MATLAB code \phipade\ associated with the algorithm does not output $\varphi_0(A)$.}
$\varphi_j$ for $j = 0\colon p$. The algorithm scales the input matrix to bring its norm close to one, computes certain powers of the scaled matrix, and reuses them to \emph{independently} evaluate the diagonal \Pant s to each matrix $\varphi$-function, resulting in distinct numerator- and denominator polynomials for different functions. Finally, it applies the double-argument formula~\eqref{double.formula} to reverse the effect of scaling.
The evaluation of the $p+1$ \Pant s constitutes the most computationally expensive component of the algorithm---each approximant evaluation additionally requires a matrix inversion and matrix multiplications, despite the reuse of computed matrix powers.

The algorithm of Al-Mohy~\cite[Alg.~4.1]{almo24} computes the exponential of block triangular matrices by exploiting their structure, without explicitly forming the full block matrix. Given the matrices $A$, $J$, and $E$ defined in Theorem~\ref{Thm.block.phi}, his algorithm simultaneously computes $\e^A$, $\e^J$, and the matrix in \eqref{D.exp.mat}. While the algorithm is effective for computing matrix exponentials, it is not specifically tailored for the evaluation of matrix $\varphi$-functions. 
As general-purpose methods, the Schur--Parlett algorithms~\cite{dahi03}, \cite{hili21}, which require either computation of the derivatives or use of variable-precision arithmetic, are not expected to be as efficient or accurate as specialized algorithms for matrix $\varphi$-functions. 

\section{Numerical experiment}
\label{sect.experiment}

This section evaluates the performance of the proposed algorithm for computing matrix $\varphi$-functions in comparison with the existing ones. The following \alg s are tested:
\begin{itemize}
\item
\phifunm: our MATLAB implementation of \Alg~\ref{alg.phi};
\item
\phipade: the implementation from the EXPINT package~\cite{bsw07}, which realizes the algorithm of Skaflestad and Wright~\cite[Alg.~1]{skwr09}, executed in two configurations:
\begin{itemize}
\item \phipadedft: the default setting, which uses the diagonal $[7/7]$ Pad\'e approximant;
\item \phipadeopt: the adaptive setting proposed in \cite[Alg.~1]{skwr09}, which seeks the optimal Pad\'e degree $m$ in the range $3 \le m \le 13$ to minimize the leading asymptotic computational cost;
\end{itemize}
\item
\expmblktri: the algorithm of Al-Mohy \cite{almo24}, designed for computing the exponential of block triangular matrices; and
\item
\expm: the MATLAB built-in function for the matrix exponential \cite{alhi09a}, intended solely for the computation of $\varphi_0(A) = \e^A$.
\end{itemize}
Optimized for IEEE double precision, \phifunm\ does not employ the mixed-precision Paterson--Stockmeyer scheme~\cite{liu25}, which is primarily relevant in variable-precision arithmetic.
The experiments were run using the 64-bit GNU/Linux version of MATLAB 24.2 (R2024b Update 3) on a desktop computer
equipped with an Intel i5-12600K processor running at 3.70 GHz
and with 32GiB of RAM. The code that produces the results in this section is available on GitHub.\footnote{\url{https://github.com/xiaobo-liu/phi\_funm}}
The test matrices consist of two sets.
\begin{enumerate}
	\item[\textbf{Set 1}] 108 \emph{nonnormal} matrices taken from the built-in groups of Anymatrix~\cite{himi21} and from a collection\footnote{\url{https://github.com/xiaobo-liu/matrices-expm}} of matrices \cite{liu25} commonly used in the matrix function literature \cite{alhi09a}, \cite{ahl22}, \cite{fahi19}.
	\item [\textbf{Set 2}] Hessenberg (tridiagonal) matrices constructed from the Arnoldi (Lanczos) process within Krylov methods for computing the action of $\varphi_j(K)$ on the vector of all ones, where the matrix $K$ is listed in Table~\ref{tab:testmat}.
\end{enumerate}
\begin{table}[t]
	\caption{Summary of test matrices from~\cite{alhi11} and the SuiteSparse collection~\cite{dahu11}.}\label{tab:testmat}
	\centering
	\scalebox{.85}{
	\begin{tabularx}{\textwidth+0.4cm}{@{\extracolsep{\fill}}@{\hspace{2pt}}l@{\hspace{.5cm}}rr@{\hspace{.4cm}}l}
		\toprule[1pt] 
		Matrix $K$  & Size & Nonzeros & Description  \\ [0.5ex] 
		\midrule
		bcspwr10 & 5,300 & 21,842 &  Power network problem\\
		gr\_30\_30 & 900 & 7,744 &	Discretization of Laplacian by a nine-point
		stencil \\
		helm2d03 & 392,257 & 2,741,935 & Helmholtz equation on a unit square \\
		orani678 &	2,529 & 90,158	 & Economic problem \\
		poisson99 & 9,801 &	48,609 & Finite diﬀerence discretization of the 2D Laplacian\\
		\bottomrule[1pt] 
	\end{tabularx}
}
\end{table}
	
For each computed matrix $\varphi_j$-function $\widehat{X}_j$ of $A$, we assess its accuracy via the normwise relative forward error $\normi{\varphi_j(A)-\widehat{X}_j} / \normi{\varphi_j(A)}$, where the reference solution $\varphi_j(A)$ is obtained by invoking the~\t{expm\_mp} function~\cite{fahi19} in 200 digits of precision for the exponential of $W$ in Theorem~\ref{Thm.block.phi} and then extracting the respective blocks. We also gauge the forward stability of the algorithms by reporting $\kappa_{\varphi_j}(A)u$, where $\kappa_{\varphi_j}(A)$ is the $1$-norm condition number~\cite[sect.~3]{high:FM} estimated by applying the \t{funm\_condest1} function of~\cite[Alg.~3.22]{high:FM} to \expm\ (for $j=0$) and \phifunm\ (for $j>0$).
\subsection{Accuracy and stability}

\begin{figure}
	\begin{subfigure}{1\linewidth}
		\centering
		\includegraphics[width=0.45\textwidth,height=3.5cm]{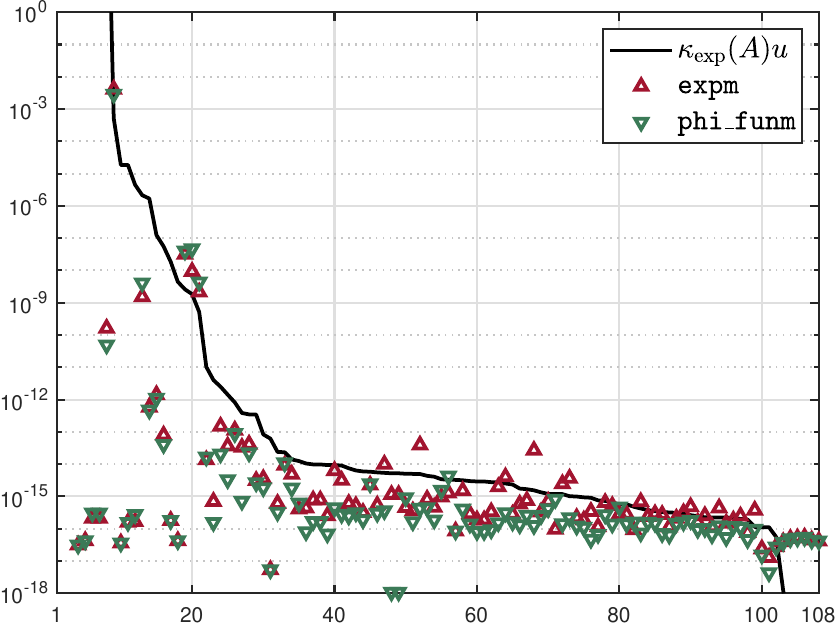}
		\hspace{1em}
		\includegraphics[width=0.45\textwidth,height=3.5cm]{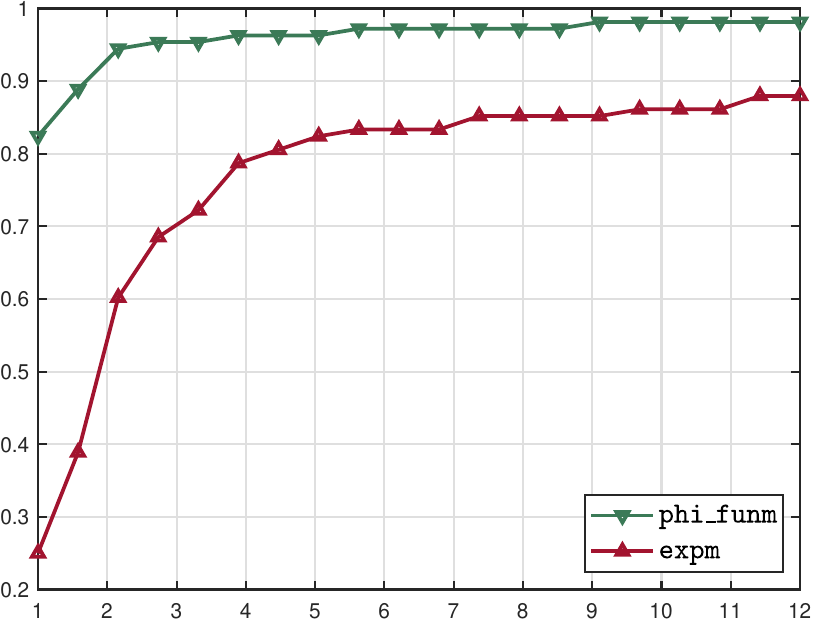}
		\caption{$j=0$.}
	\end{subfigure}
	\begin{subfigure}{1\linewidth}
		\centering
		\includegraphics[width=0.45\textwidth,height=3.5cm]{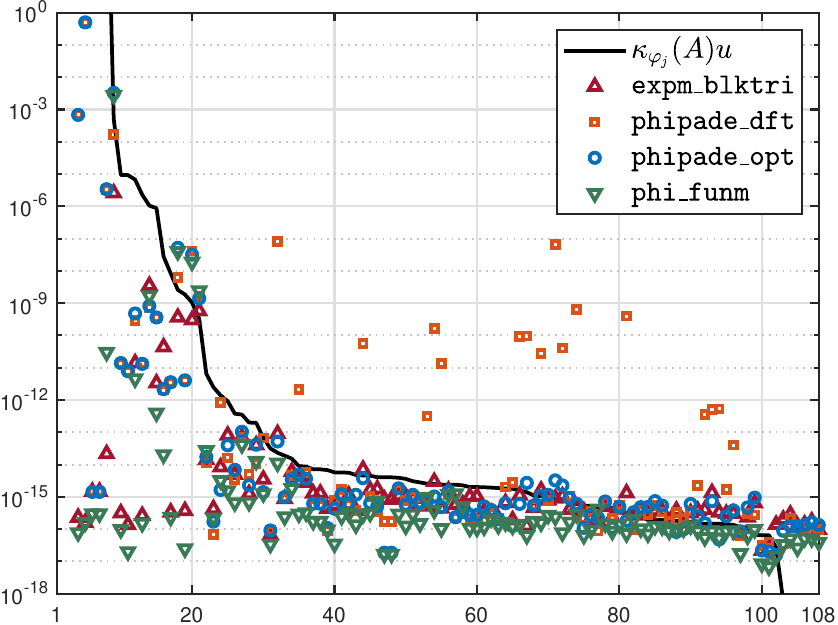}
		\hspace{1em}
		\includegraphics[width=0.45\textwidth,height=3.5cm]{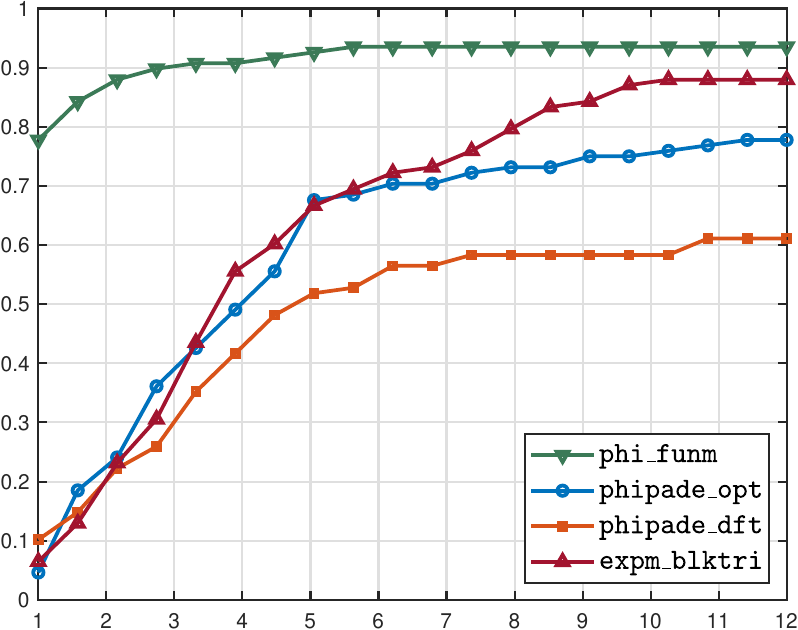}
		\caption{$j=1$.}
	\end{subfigure} \\
	\begin{subfigure}{1\linewidth}
		\centering
		\includegraphics[width=0.45\textwidth,height=3.5cm]{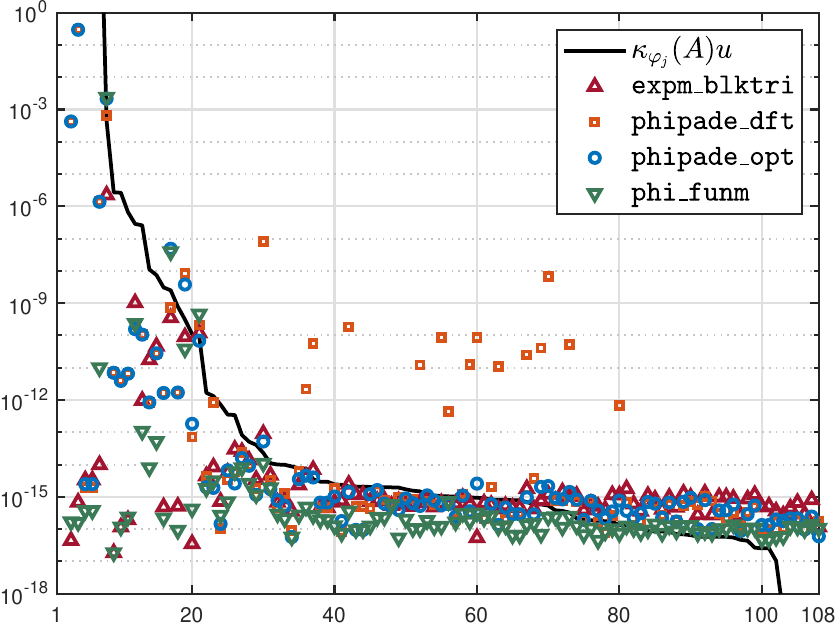}
		\hspace{1em}
		\includegraphics[width=0.45\textwidth,height=3.5cm]{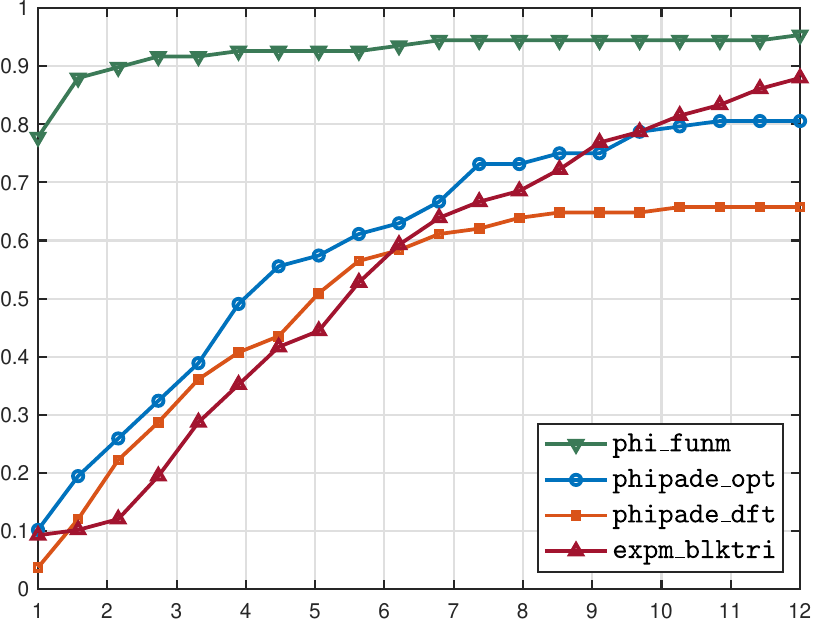}
		\caption{$j=4$.}
	\end{subfigure} \\
	\begin{subfigure}{1\linewidth}
		\centering
		\includegraphics[width=0.45\textwidth,height=3.5cm]{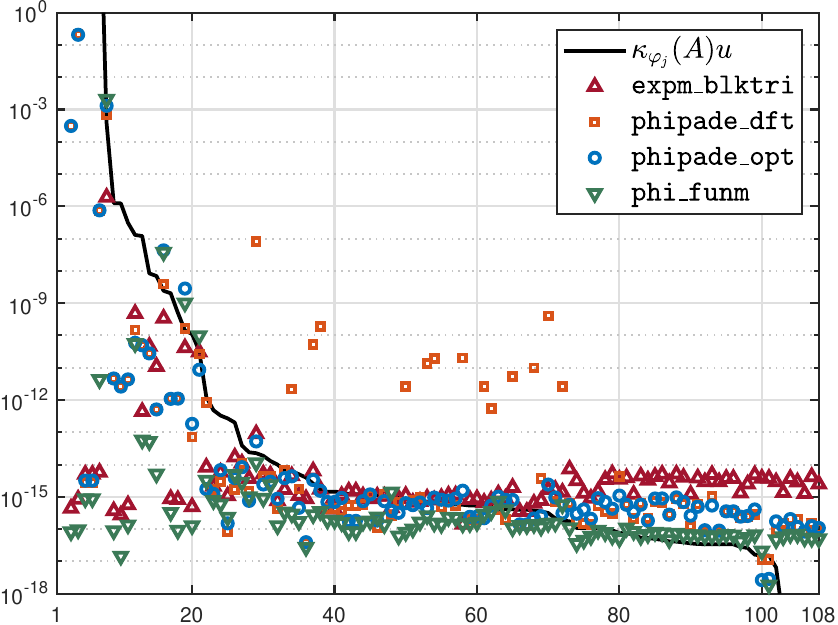}
		\hspace{1em}
		\includegraphics[width=0.45\textwidth,height=3.5cm]{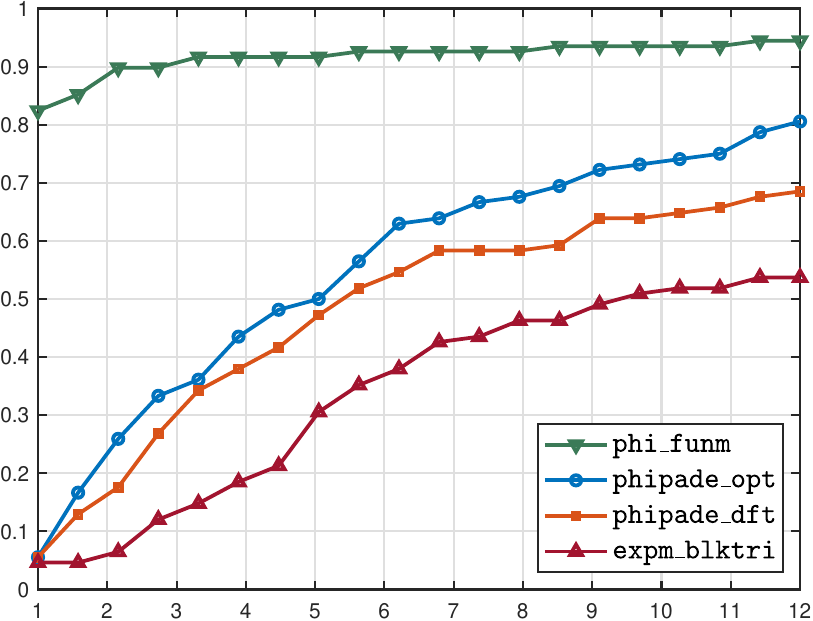}
		\caption{$j=7$.}
	\end{subfigure} \\
	\begin{subfigure}{1\linewidth}
		\centering
		\includegraphics[width=0.45\textwidth,height=3.5cm]{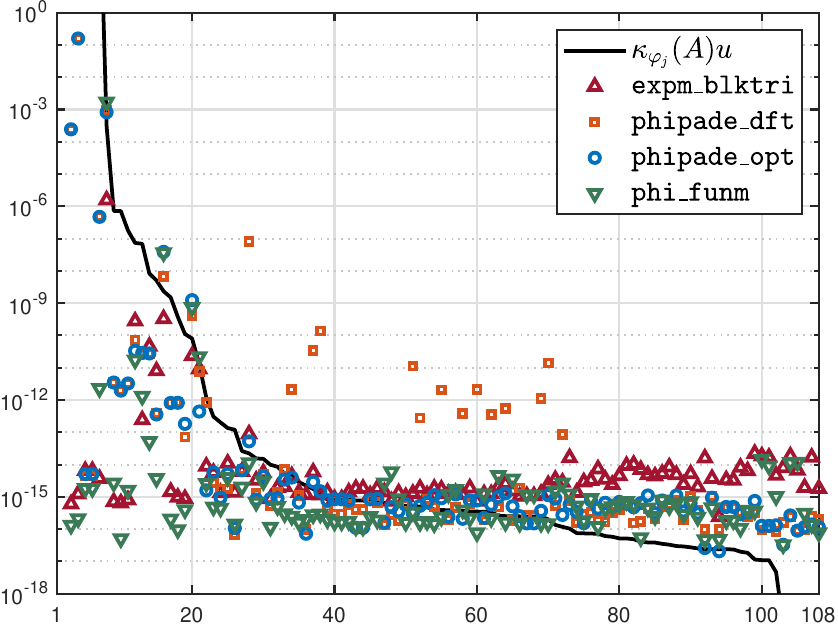}
		\hspace{1em}
		\includegraphics[width=0.45\textwidth,height=3.5cm]{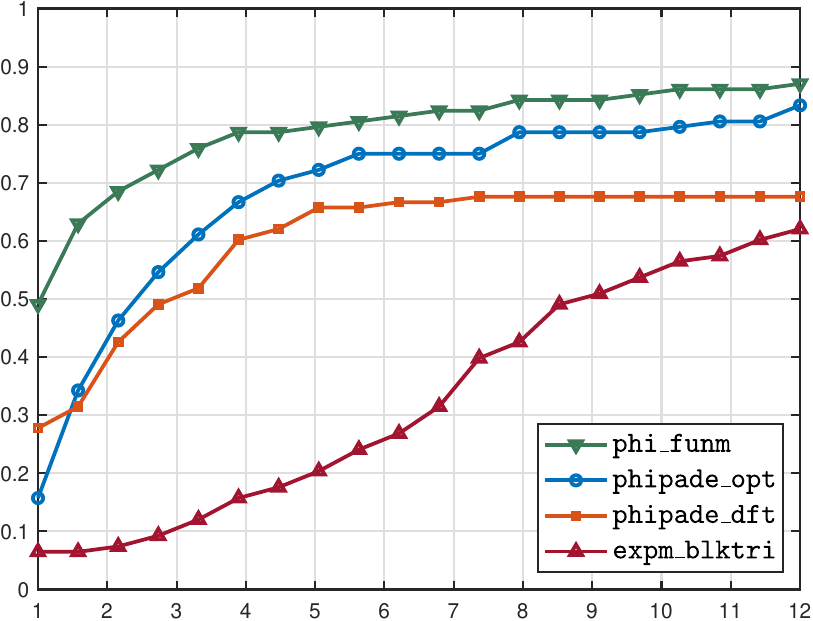}
		\caption{$j=10$.}
	\end{subfigure}
	\vspace{-2em}
	\caption{Relative forward errors and corresponding performance profiles of the algorithms for computing $\varphi_j(A)$.}
	\label{fig.test-main-accuracy}
\end{figure}
In the first experiment, the test matrices are from Set 1 and have dimensions from $2\times 2$ to $41\times 41$, and most of them are real matrices of size $20\times 20$.

When a linear combination of the matrix $\varphi$-functions, as in the form of~\eqref{phi.lin.comb}, is required, all the algorithms aiming for $\varphi_p(A)$ can compute the matrix $\varphi$-functions simultaneously, while the variants of \phipade\ do not produce $\varphi_0(A)$.
Given the largest index $p = 10$, Figure~\ref{fig.test-main-accuracy} presents the relative forward errors
\[
\frac{\normi{\varphi_j(A) - \widehat{X}_j}}{\normi{\varphi_j(A)}}, \quad j \in \{0, 1, 4, 7, 10\},
\]
sorted in descending order of the condition number $\kappa_{\varphi_j}(A)$, together with the corresponding performance profiles~\cite{domo02}.
In the performance profiles, the $y$-coordinate of a given algorithm represents the frequency of test matrices for which its relative error is within a factor $\beta$ of the smallest error among all algorithms, where $\beta$ is the $x$-coordinate.

The results clearly show the superior accuracy of \phifunm\ over its competitors, especially for matrix $\varphi$-functions with small index $j$. 
It is also noteworthy that the algorithm of~\cite{skwr09} in its default setting (\phipadedft) is highly unstable for a number of well-conditioned problems, with errors exceeding the stability threshold by several orders of magnitude.

For $\varphi_0(A)=\e^A$, the comparison is made between \phifunm\ and \expm, in which case the former is distinctly more accurate than the latter. This is largely because \phifunm\ indirectly evaluates the $[m+p/m]$ \Pant\ to the exponential, which is of higher degree than the $[m/m]$ \Pant\ evaluated by \expm. It also reveals that the recurrence~\eqref{rec.pade} has been computed to high relative accuracy.

For $\varphi_j(A)$ with $j>0$, the stability of \phipadedft\ is hindered by its use of fixed-degree \Pant, which can lead to potential overscaling issues. 
In contrast, \phipadeopt, which allows flexibility in the parameter selection, exhibits much better forward stability.
The general-purpose algorithm \expmblktri\ delivers good accuracy when the index $j$ is small, but its performance deteriorates as $j$ grows. 
This is perhaps unsurprising, as the algorithm operates on larger and sparser matrices without exploiting the special structure within the assembled blocks.

\subsection{Asymptotic computational cost}
\label{sect.asymp-cost}

\begin{figure}
	\centering
	\begin{subfigure}{0.45\linewidth}
		\centering
		\includegraphics[width=\linewidth,height=4.5cm]{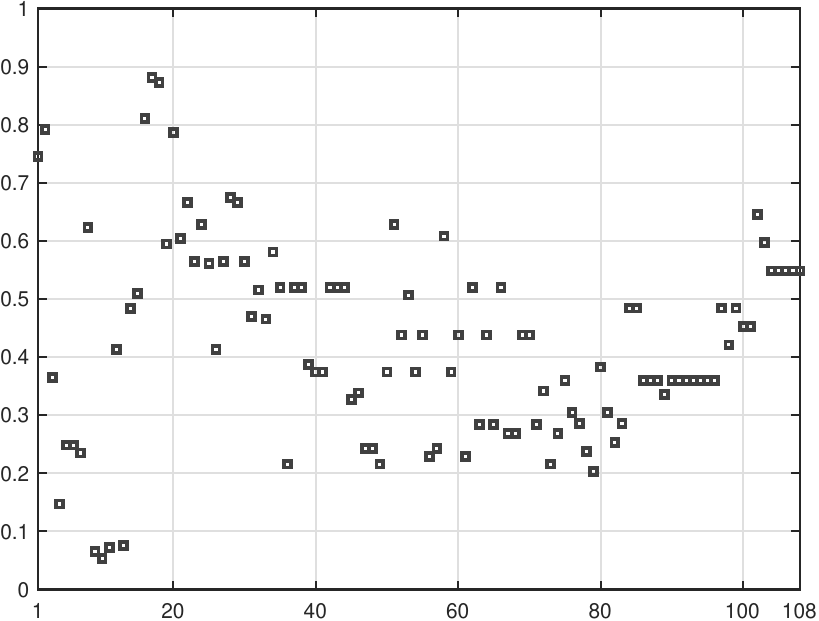}
		\caption{\phifunm\ over \phipadedft.}
	\end{subfigure}
	\hspace{1em}
	\begin{subfigure}{0.45\linewidth}
		\centering
		\includegraphics[width=\linewidth,height=4.5cm]{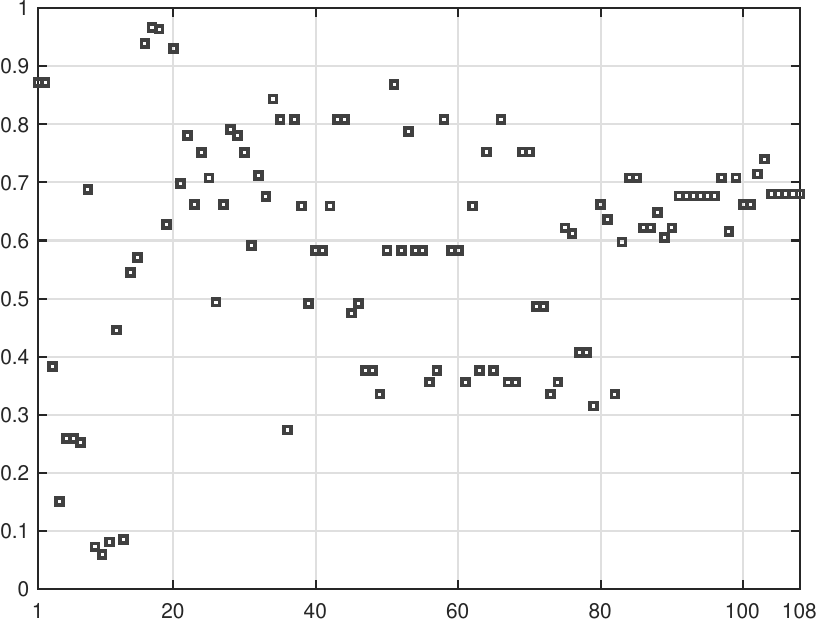}
		\caption{\phifunm\ over \phipadeopt.}
	\end{subfigure}
	\caption{Ratios of the asymptotic computational cost of \phifunm\ to \phipade.}
	\label{fig.test-main-cost}	
\end{figure}

Corresponding to the accuracy comparison presented in Figure~\ref{fig.test-main-accuracy}, the asymptotic computational costs of \phifunm\ and the variants of \phipade\ are measured in terms of the equivalent number of matrix products, as shown in Figure~\ref{fig.test-main-cost}.

It is observed that \phifunm\ is more efficient than the two variants of \phipade\ in every case, often reducing the cost by more than half.
Compared with \phipadedft, the cost-optimized variant \phipadeopt, has indeed narrowed the efficiency gap in many cases, but the advantage of \phifunm\ remains evident, as it still achieves costs that are $10\times$ to $20\times$ lower in several cases. 

This superiority of \phifunm\ in efficiency is mainly due to two improvements.
First, it performs the scaling in reliance on the $\a_r$-based sequence~\eqref{def-alphap} rather than $\normi{A}$, which makes it less prone to the overscaling issue and reduces the cost. 
Second, unlike \phipade, which evaluates \Pant s of the same degree to all $\varphi_j$ and thus produces varying numerators and denominators for different $\varphi_j$-functions, \phifunm\ adapts the \Pant\ degree per $\varphi_j$ while keeping the denominator polynomial fixed (see Corollary~\ref{Corol.pade}), saving at least $p$ invocations of the multiple linear systems solver.

\subsection{Runtime comparison and code profiling}
Building upon the asymptotic complexity analysis presented in section~\ref{sect.asymp-cost}, we now compare the execution time of \phifunm\ and the two variants of \phipade. 
To better understand the computational characteristics of \phifunm, we also perform its code profiling and report the execution time breakdown on problems of varying sizes.

In the runtime comparison, we use 55 test matrices from Set 1, with variable sizes parametrized by $n$.
We compare the execution time of the algorithms on these matrices with different dimensions.
With $n=20$, all algorithms take about $10^{-3}$ seconds, so the runtime difference, which is of at most one order of magnitude, is insignificant, and such small scale computations in MATLAB are often dominated by interpreter overhead.
We increase the problem size to $n=200, 500, 2500$ and compare the ratios of the execution time of \phifunm\ to the two variants of \phipade, respectively. Figure~\ref{fig.test-time} presents the results.
For problems of size $O(100)$, \phifunm\ and the two variants of \phipade\ show comparable speed in most cases. As the problem size grows, however, the lower asymptotic cost of \phifunm\ is reflected in the execution time: when $n=2500$, it is faster than \phipadedft\ or \phipadeopt\ in over 80\% of cases.
Notably, \phifunm\ is also more reliable, with runtimes never exceeding roughly twice that of the fastest algorithm, whereas its competitors can be up to $16\times$ slower.

\begin{figure}
	\centering
	\begin{subfigure}{0.45\linewidth}
		\centering
		\includegraphics[width=\linewidth,height=4.5cm]{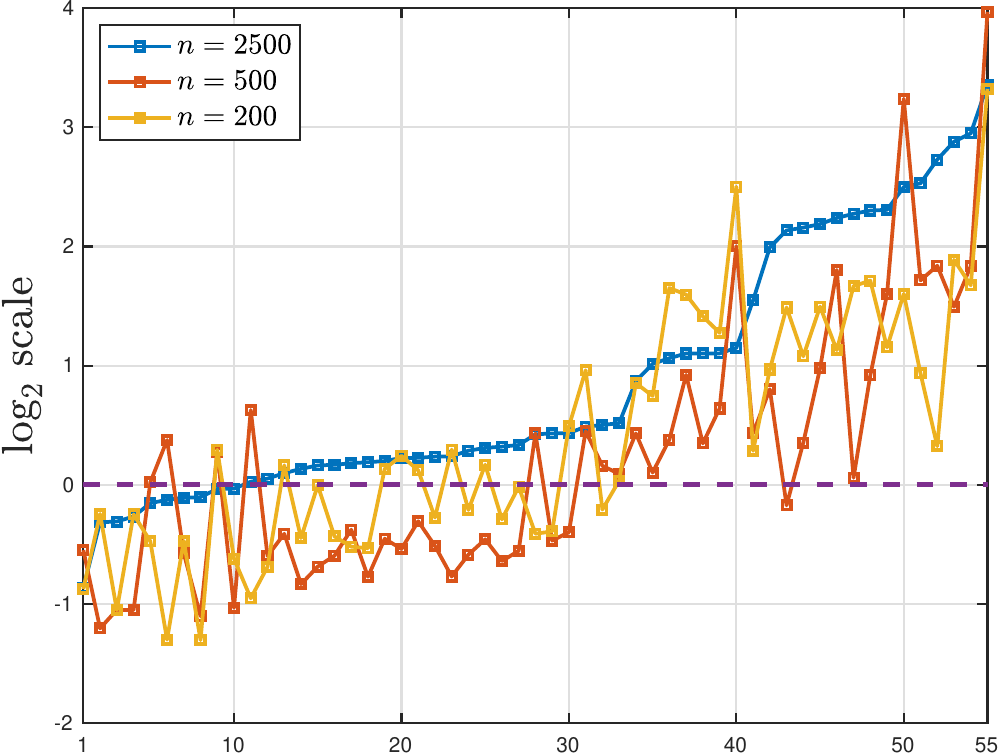}
		\caption{\phifunm\ vs \phipadedft.}
	\end{subfigure}
	\hspace{1em}
	\begin{subfigure}{0.45\linewidth}
		\centering
		\includegraphics[width=\linewidth,height=4.5cm]{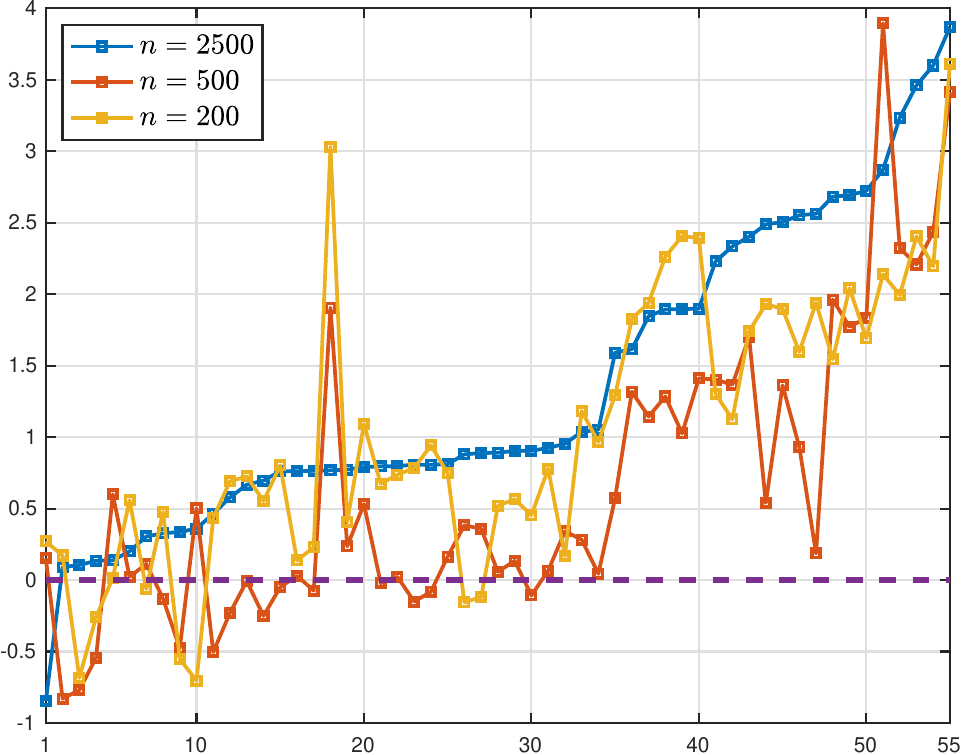}
		\caption{\phifunm\ vs \phipadeopt.}
	\end{subfigure}
	\caption{$\log_2$-speedup of \phifunm\ relative to \phipade. Positive values indicate \phifunm\ is faster.}
	\label{fig.test-time}	
\end{figure}

\begin{table}
	\caption{\phifunm\ profiling. Matrix multiplications in \Pant\ ($M_{\mathrm{eval}}$) and recovery phase ($M_{\mathrm{recv}}$), total runtime ($T_{\mathrm{tot}}$) in seconds, and time percentages for parameter selection ($P_{\mathrm{par}}$), \Pant\ ($P_{\mathrm{eval}}$), and recovery phase ($P_{\mathrm{recv}}$).}

	\label{tab:profiling}
	\centering
	\scalebox{.85}{
		\begin{tabularx}{0.85\textwidth}{@{\extracolsep{\fill}}r@{\hspace{1.0cm}}rrrrrrr}
			\toprule[1pt]
			& $n$ & $M_{\mathrm{eval}}$ & $M_{\mathrm{recv}}$ & $P_{\mathrm{par}}$ & $P_{\mathrm{eval}}$ & $P_{\mathrm{recv}}$ & $T_{\mathrm{tot}}$ \\
			\midrule
			\verb|A| &  20 &  17 & 55 & 49.4\% & 33.5\% & 17.1\% &  0.0  \\
			& 200 &  16 & 132 & 4.0\% & 12.1\% & 83.9\% &  0.1  \\
			& 500 &  16 & 165 & 0.7\% & 9.3\% & 90.1\% &  0.6  \\
			& 2500 &  17 & 209 & 0.8\% & 6.3\% & 92.8\% & 43.9  \\
			\midrule
			\verb|B| &  20 &  16 & 11 & 62.6\% & 30.8\% & 6.6\% &  0.0  \\
			& 200 &  16 & 55 & 8.4\% & 15.9\% & 75.7\% &  0.0  \\
			& 500 &  16 & 66 & 1.9\% & 18.0\% & 80.1\% &  0.3  \\
			& 2500 &  17 & 88 & 1.8\% & 13.3\% & 84.8\% & 20.4  \\
			\midrule
			\verb|C| &  20 &  15 &  0 & 55.1\% & 44.9\% & 0.0\% &  0.0  \\
			& 200 &  16 &  0 & 19.4\% & 80.6\% & 0.0\% &  0.0  \\
			& 500 &  17 &  0 & 1.4\% & 98.6\% & 0.0\% &  0.5  \\
			& 2500 &  16 & 11 & 4.1\% & 54.7\% & 41.2\% & 17.5  \\
			\bottomrule[1pt]
		\end{tabularx}
	}
\end{table}

Table~\ref{tab:profiling} reports the execution time breakdown of \phifunm\ on three classes of matrices 
\begin{verbatim}
	A = anymatrix('gallery/circul', n);		   % circulant matrix
	B = anymatrix('gallery/triw', n, -2);  % upper triangular matrix
	C = anymatrix('core/vand', n);         % Vandermonde matrix
\end{verbatim}
where $n$ ranges from 20 to 2500. The first matrix is a circulant matrix whose first row contains the integers from $1$ to $n$, making its $1$-norm increase quadratically with $n$. The second matrix is upper triangular and has a condition number increasing rapidly with $n$. The third matrix is a Vandermonde matrix based on equally spaced points on $[0,1]$, so its $1$- and $\infty$ norms are equal to $n$.

The cost of parameter selection of \phifunm, which includes the $O(n^2)$ computational overhead from the norm estimations~\cite{hiti00n}, is typically the most expensive part for small matrices, but its weight becomes increasingly negligible as problem size grows. Moreover, the percentages $P_{\mathrm{eval}}$ and $P_{\mathrm{recv}}$ scale more consistently with the number of matrix multiplications $M_{\mathrm{eval}}$ and $M_{\mathrm{recv}}$, respectively, as the dimension increases.

\subsection{Hessenberg matrices from Krylov methods}
Finally, we examine the algorithms on Hessenberg matrices arising in the Krylov method for the action of $\varphi$-functions on operand vectors. 
The test matrices are generated from Set 2 via the Arnoldi iteration using the \texttt{arnoldi} routine from the Matrix Function Toolbox~\cite[App.~D]{high:FM}.
These matrices have been used in the literature that targets at accelerating exponential integrators~\cite{alhi11}, \cite{niwr12}. 
We use the Krylov subspace dimension $m=30$, as used in~\cite{niwr12}, as well as $m=80$, which might be required for very large and stiff systems.

\begin{table}
	\caption{Relative forward errors and computational cost (matrix multiplication equivalents) for \phifunm, \phipadedft, and \phipadeopt\ on $m\times m$ Hessenberg matrices from Set~2, for $\varphi_1$ and $\varphi_4$.}
	\label{tab:hess}
	\centering
	\setlength{\tabcolsep}{5pt}
	\scalebox{.78}{
			\begin{tabularx}{1.25\textwidth}{ll lc lc lc lc lc}
				\toprule[1pt] 
				\multicolumn{2}{c}{} & \multicolumn{2}{c}{ \verb|bcspwr10|} & \multicolumn{2}{c}{ \verb|gr_30_30|} & \multicolumn{2}{c}{ \verb|helm2d03|} & \multicolumn{2}{c}{ \verb|orani678|} & \multicolumn{2}{c}{ \verb|poisson99|}  \\[3pt]
				\multicolumn{2}{c}{$p=1$} & \hspace{3.5pt}Error & Cost & \hspace{3.5pt}Error & Cost & \hspace{3.5pt}Error & Cost & \hspace{3.5pt}Error & Cost & \hspace{3.5pt}Error & Cost  \\
				\midrule
				\multirow{3}{*}{\rotatebox{0}{$m = 30$}} & \phifunm & 3.2e-15 & 11.3 & 1.0e-15 & 12.3 & 1.4e-15 & 12.3 & 3.7e-16 & \phantom{1}8.3 & 7.5e-14 & 34.3 \\
				& \phipadedft & 5.3e-9 & 14.7 & 1.1e-9 & 16.7 & 1.9e-10 & 16.7 & 5.7e-16 & 14.7 & 8.2e-14 & 38.7 \\
				& \phipadeopt & 2.0e-15 & 13.7 & 3.3e-15 & 15.7 & 2.1e-15 & 15.7 & 7.4e-16 & 13.7 & 7.8e-14 & 35.7 \\
				\midrule
				\multirow{3}{*}{\rotatebox{0}{$m = 80$}} & \phifunm & 3.2e-15 & 11.3 & 1.0e-15 & 12.3 & 1.5e-15 & 12.3 & 4.9e-16 & \phantom{1}8.3 & 9.1e-14 & 34.3 \\
				& \phipadedft & 5.3e-9 & 14.7 & 3.4e-14 & 18.7 & 1.9e-10 & 16.7 & 6.4e-16 & 18.7 & 1.1e-13 & 38.7 \\
				& \phipadeopt & 2.0e-15 & 13.7 & 1.8e-15 & 15.7 & 2.1e-15 & 15.7 & 9.8e-16 & 15.7 & 9.9e-14 & 35.7 \\
				\midrule
				\multicolumn{2}{c}{$p=4$} & \hspace{3.5pt}Error & Cost & \hspace{3.5pt}Error & Cost & \hspace{3.5pt}Error & Cost & \hspace{3.5pt}Error & Cost & \hspace{3.5pt}Error & Cost  \\
				\midrule
				\multirow{3}{*}{\rotatebox{0}{$m = 30$}} & \phifunm & 1.1e-15 & 16.3 & 8.2e-15 & 17.3 & 4.0e-15 & 17.3 & 6.8e-16 & 10.3 & 1.5e-14 & 72.3 \\
				& \phipadedft & 5.0e-10 & 27.7 & 1.1e-9 & 32.7 & 1.8e-10 & 32.7 & 3.6e-16 & 27.7 & 5.4e-14 & 87.7 \\
				& \phipadeopt & 1.3e-15 & 23.7 & 3.1e-15 & 28.7 & 1.9e-15 & 28.7 & 4.5e-16 & 23.7 & 5.2e-14 & 78.7 \\
				\midrule
				\multirow{3}{*}{\rotatebox{0}{$m = 80$}} & \phifunm & 1.1e-15 & 16.3 & 8.8e-15 & 17.3 & 4.1e-15 & 17.3 & 8.6e-16 & 10.3 & 2.0e-14 & 72.3 \\
				& \phipadedft & 5.0e-10 & 27.7 & 3.3e-14 & 37.7 & 1.8e-10 & 32.7 & 1.3e-15 & 37.7 & 6.4e-14 & 87.7 \\
				& \phipadeopt & 1.3e-15 & 23.7 & 1.6e-15 & 28.7 & 1.9e-15 & 28.7 & 1.1e-15 & 28.7 & 5.9e-14 & 78.7 \\
				\bottomrule[1pt]
			\end{tabularx}
	}
\end{table}


The results in Table~\ref{tab:hess} show a similar trend to the previous experiments. The new algorithm \phifunm\ consistently incurs lower computational cost than the two variants of \phipade.
We also examined the execution times of these algorithms (not reported) and found them to be typically between $10^{-3}$ and $10^{-2}$ seconds.
Both \phifunm\ and \phipadeopt\ deliver good and comparable accuracy, whereas \phipadedft\ again exhibits instability in several cases.

\section{Conclusions}
\label{sect.conclusion}
We have developed a novel algorithm for the simultaneous computation of matrix $\varphi$-functions, which play a central role in exponential integrator methods for solving stiff systems of ODEs. The proposed algorithm builds on a carefully designed scaling and recovering method.

The key strengths of the algorithm lie, first, in its rigorous backward error analysis, which yields sharp relative error bounds in terms of the sequence $\norm{A^k}^{1/k}$, enabling the selection of the smallest possible scaling parameter. Second, the implementation of the recurrence relation~\eqref{rec.pade} eliminates the need for repeated rational approximations: the highest-index function $\varphi_p$ is approximated using a diagonal Pad\'e approximant, and the lower-index functions $\varphi_j$, for $0 \le j < p$, are then efficiently computed via essentially a single matrix multiplication for each $j$. The algorithmic parameters are selected on the fly to optimize the overall computational cost.

Another important feature of the algorithm is its ability to exploit matrix triangularity. When the input matrix is triangular or quasi-triangular, as commonly occurs after a Schur decomposition, the recovery phase effectively controls error propagation in computing the matrix exponential, mitigating the transfer of errors  to the other $\varphi$-functions.
Leveraging this feature, if the input is a Hessenberg matrix produced by a Krylov algorithm, one can first compute its Schur form and then apply the proposed algorithm to the resulting (quasi-)triangular factor.

A comprehensive set of numerical experiments demonstrates the consistent performance advantages of the proposed algorithm over existing alternatives, in both computational efficiency and numerical accuracy.
The algorithm exhibits remarkable numerical forward stability; a full characterization of its overall numerical backward stability remains an interesting open problem that we look forward to addressing in future work.
\section*{Acknowledgments}  
We thank the anonymous reviewers for their comments and suggestions, which helped improve the presentation of this paper.

\bibliographystyle{siamplain}
\bibliography{strings,phi_paper,njhigham}


\end{document}